\documentclass[11pt]{article} % Include statements
\usepackage[left=1in,top=1in,right=1in,bottom=1in,nohead]{geometry}

\usepackage{paralist}
\usepackage{subcaption}

\usepackage{amsfonts,amssymb,amsmath,amsthm}

\usepackage{pdfsync}
\usepackage{hyperref}
\usepackage{graphicx}
\usepackage[sort,sectionbib,round]{natbib}
\hypersetup{backref,colorlinks=true,citecolor=blue,linkcolor=blue,urlcolor=blue}

\usepackage{aliascnt}

\usepackage{lassoOurDefs}
\newcommand{\cConst}{c} % for notation throughout the proof
\newcommand{\cConstVal}{1/8e^2} % for giving the correct value below
\newcommand{\cConstValFrac}{/8e^2} % for use in our 2 main theorems
\newcommand{\email}[1]{\href{mailto:#1}{#1}}

\begin{document}

\title{Risk Consistency of Cross-Validation with Lasso-Type Procedures}
\author{Darren Homrighausen\\
  Department of Statistics\\Colorado State University\\ Fort
  Collins, CO 80523\\\email{darrenho@stat.colostate.edu} 
  \and Daniel J. McDonald \\
  Department of Statistics\\ Indiana University \\ Bloomington,
  IN 47408\\\email{dajmcdon@indiana.edu}
}

\maketitle

\begin{abstract}
  The lasso and related
  sparsity inducing algorithms have been the target of substantial
  theoretical and applied 
  research. Correspondingly, many results are known about their
  behavior for a fixed or optimally chosen tuning parameter
  specified up to unknown constants. In
  practice, however, this oracle tuning parameter is inaccessible
  so one  must use the data to select one. Common statistical practice is to use
  a variant of cross-validation for this task.  However,
  little is known about the theoretical properties of the
  resulting predictions with such
  data-dependent methods.  We consider
  the high-dimensional setting with random design wherein the number
  of predictors $p$ grows with the number of
  observations $n$. Under typical assumptions on the data generating
  process, similar to those in the literature, we recover oracle rates
  up to a log factor when choosing the tuning parameter with
  cross-validation. Under weaker conditions, when the true model is
  not necessarily linear, we show that the 
  lasso remains risk consistent relative to its linear
  oracle. We also
  generalize these results to the group lasso and
  square-root lasso and investigate the predictive and model selection performance of
  cross-validation via simulation.
\vspace{9pt}
\noindent {\it Key words and phrases:}
  Linear oracle; Model selection; Persistence;
  Regularization
\end{abstract}

\section{Introduction}
\label{sec:intro}

Following its introduction in the
statistical~\citep*{Tibshirani1996} and signal
processing~\citep*{ChenDonoho1998} communities, $\ell_1$-regularized
linear regression has become a fixture as both a data analysis tool
% \citet*{LeeZhu2010,ShiWahba2008} 
and as a subject for deep theoretical
investigations.
% \citet*{FuKnight2000,GreenshteinRitov2004,MeinshausenBuhlmann2006} 
In particular, for a response vector $Y \in \mathbb{R}^n$, design matrix
$\mathbb{X} \in \mathbb{R}^{n \times p}$, and tuning parameter $\lambda$, we
consider the lasso problem of finding
\begin{equation}
  \label{eq:lagrangian}
  \hat{\beta}_\lambda = \argmin_\beta \frac{1}{n}\norm{Y-\X\beta}_2^2 +
  \lambda\norm{\beta}_1 
\end{equation}
for any $\lambda$, where
$\norm{\cdot}_2$ and $\norm{\cdot}_1$ indicate the Euclidean and 
$\ell_1$-norms respectively.
An equivalent but less computationally convenient specification of the
lasso problem given by Equation \eqref{eq:lagrangian} is the
constrained optimization version:
\begin{equation} \betat{t} := \hat\beta(\B_t) \in
  \argmin_{\beta\in\B_t} \frac{1}{n}\norm{Y - \X\beta}_2^2
  \label{eq:optimization}
\end{equation} where $\B_t := \{ \beta : \norm{\beta}_1 \leq t \}$. 
By convexity, for each $\lambda$ (or $t$), there is
always at least one solution to these optimization problems.  While it is
true that the solution is not necessarily unique if 
rank$(\mathbb{X}) < p$, this detail is unimportant for our
purposes, and we abuse notation slightly by referring to
$\betat{\lambda}$ as
`the' lasso solution.  

These two optimization problems are equivalent mathematically, but they
differ substantially in practice.  Though the constrained optimization problem 
(\refeq{eq:optimization}) can be solved via quadratic programming,
most algorithms use the Lagrangian formulation (\refeq{eq:lagrangian}).  In this paper,
we address both estimators as each is more amenable to theoretical
treatment in different contexts.

The literature contains numerous results regarding the statistical
properties of the lasso, and, while it is beyond the scope of this paper to give a
complete literature review, we highlight some of these results here.
Early results about the asymptotic distribution of the lasso
solution are shown in \citet*{FuKnight2000} under the assumption
that the sample covariance matrix has a nonnegative definite
limit and $p$ is fixed. Many authors \citep*[e.g.][]{DonohoElad2006,MeinshausenBuhlmann2006,MeinshausenYu2009,
Wainwright,ZhaoYu2006,Zou2006} have investigated model selection properties of
the lasso---showing that when the best predicting model is linear and sparse,
the lasso will tend to asymptotically recover those predictors. The literature 
has considered this setting under various ``irrepresentability''
conditions which ensure that the relevant predictors are not too
highly correlated with irrelevant ones. \citet*{bickel2009simultaneous}
analyze the lasso and the Dantzig 
selector \citep*{CandesTao2007} under restricted eigenvalue conditions
with an oracle tuning parameter. Finally,
\citet*{BelloniChernozhukov2014} develop results for a related method,
the square-root lasso, with heteroscedastic noise and oracle tuning
parameter. 

Theoretical results such as these and others,
depend critically on the choice of tuning parameters and are typically
of the form: if $t = t_n = o(n/\log p )^{1/4}$,  then $\betat{t_n}$
possesses some desired behavior (correct model selection, risk consistency,
sign consistency, et cetera). However comforting results of this type are,
this theoretical guidance says little about the properties of the
lasso when the tuning parameter is chosen in a data-dependent way.

In this paper, we show that the lasso under random design with
cross-validated tuning parameter is indeed risk consistent under some
conditions on the joint distribution of the design matrix
$\mathbb{X}$ and the response vector $Y$. Additionally, we demonstrate that
our framework can be used to show similar results for other
lasso-type methods. Our results build on the previous theory
presented in \citet*{HomrighausenMcDonald2014} and \citet*{homrighausen2013lasso}.
 \citet*{HomrighausenMcDonald2014} proves risk
consistency for cross-validation under strong conditions on the data
generating process, most notably $n > p$, and on the cross-validation
procedure (requiring leave-one-out CV). The results in this paper
differ from those in \citet*{homrighausen2013lasso} in a number of
ways. The current paper examines the Lagrangian 
formulation of the lasso problem under typical conditions, weakens the
conditions on an upper bound for $t$, provides more refined results via
concentration inequalities, examines the influence of $K$, and
includes related results for the group lasso and the square-root lasso.

\subsection{Overview of results}
\label{sec:overview-results}

The criterion we focus on for this paper is risk consistency,
(alternatively known as persistence).  That is, we investigate the
difference between the prediction risk of the lasso estimator with
tuning parameter estimated by cross-validation and the risk of the best
linear oracle predictor (with oracle tuning parameter). Risk
consistency of lasso has previously been 
studied by
\citet*{GreenshteinRitov2004,BuneaTsybakov2007,Geer2008} and \citet*{bartlett2012}. Their
results, in contrast to ours,
assume that the tuning parameter is selected in an oracle
fashion. 

We present two main results which make progressively stronger
assumptions on the data generating process and use both forms of the
lasso estimator. The first imposes strong conditions
on the design matrix, assumes the linear model is true, and that this
linear model is sparse. The second is more general and allows
the true model to be neither linear nor 
correctly specified. For this reason, our focus is on risk consistency rather than
estimation of a ``true'' parameter or correct
identification of a ``true'' sparsity pattern. Additionally,
well-known results of \citet*{Shao1993} imply that cross-validation
leads to inconsistent model selection in general, suggesting that
results for sparse recovery may not exist. Although prediction is an
important goal, one is often interested in variable
selection for more interpretable models or follow-up
experiments. In light of the negative results in \citet*{Shao1993}, we
are unable to
offer theoretical results which promise consistent model selection by
cross validation, but simulations in \autoref{sec:simulations} suggest
that it performs well nevertheless. Both the
estimation and sparse recovery settings are frequently studied in the
literature assuming the tuning parameter is the oracle and
that the data generating model is linear
\citep*[e.g.][]{BuneaTsybakov2007,CandesPlan2009,DonohoElad2006,LengLin2006,MeinshausenBuhlmann2006,MeinshausenYu2009}. 

In the first case, when the truth is linear, we examine the Lagrangian
formulation in Equation \eqref{eq:lagrangian}. In this scenario, we
prove convergence rates which differ only by a $\log$ factor relative
to those achievable with the oracle tuning parameter
\citep*[e.g.][]{NegahbanRavikumar2012,BuhlmannGeer2011,BuneaTsybakov2007}. That
is, for an $s^*$-sparse 
linear model with restricted isometry conditions on the covariance of
the design, the risk of the cross-validated estimator approaches the
oracle risk at a rate of $O(s^*\log(p)\log(n)/n)$. Under more general conditions,
we follow the approach of \citet*{GreenshteinRitov2004} and examine the
constrained optimization form in Equation \eqref{eq:optimization}. Using
our methods, we require that the set of candidate
predictors, $\B_{t_n}$, satisfies $t_n=o\left(n^{1/4}/(m_n (\log
p)^{1/4+1/(2q))}\right)$ where $m_n$ is a sequence which approaches infinity
and $q$ characterizes the tail behavior of the data. This
is essentially as quickly as one could hope relative to
\citet*{GreenshteinRitov2004} under our more general assumptions
on the design matrix. We note however that, using
empirical process techniques, \citet*{bartlett2012} have been able to
improve the rate shown in \citet*{GreenshteinRitov2004} to $t_n=o\left(n^{1/2}/(\log^{3/2} n
    \log^{3/2} (np))\right)$ for sub-Gaussian design and oracle tuning
  parameter.

% The main contribution of this paper is to show that
% when cross-validation is used to choose the tuning parameter, the lasso remains persistent
% relative to the theoretically optimal, but empirically unavailable,
% non-stochastic oracle choice.  

\subsection{Tuning parameter selection methods and outline
  of the paper}
\label{sec:related-literature}

There are several proposed data-dependent techniques for choosing $t$ or $\lambda$.  \citet*{kato2009degrees}
and \cite{TibshiraniTaylor2012}
investigate estimating the ``degrees of
freedom'' of a lasso solution.
With an unbiased estimator of the degrees of freedom, the tuning parameter can be selected by
minimizing the empirical risk penalized by a function of this estimator. 
Note that this approach requires an estimate of the variance, which is 
nontrivial when $p > n$
\citep*{giraud2012high,HomrighausenMcDonald2015a}.
Another risk estimator is the adapted Bayesian
information criterion
proposed by \citet*{WangLeng2007} which uses a plug-in estimator based
on the
second-order Taylor's expansion of the risk.
\citet*{arlot2009data} and \citet*{Saumard2011} consider ``slope
heuristics'' as a method for penalty selection with Gaussian
noise, paying particular attention to the regressogram estimator in
the first case and piecewise polynomials with $p$ fixed in the
second. Finally, \citet*{sun2012scaled} present an algorithm to jointly
estimate the regression coefficients and the noise level. This results
in a data-driven value for the tuning parameter which possesses oracle
properties under some regularity conditions. 
% \attn{See \autoref{sec:simulations} for
%a further discussion and comparison of some of these methods.}

However, many authors
(e.g.\ \citealp{EfronHastie2004,FriedmanHastie2010,GreenshteinRitov2004,
Tibshirani1996,Tibshirani2011,ZouHastie2007} and as discussed by
\citealp[][Section 2.4.1]{VanDeGeer2011}) recommend 
selecting $t$ or $\lambda$ in the lasso problem by minimizing a
$K$-fold cross-validation estimator of the risk (see 
\autoref{sec:notation-assumptions} for the precise definition).
 Cross-validation is generally well-studied in
a number of contexts, especially model selection and risk
estimation. In the context of model selection,
\citet*{ArlotCelisse2010} give a detailed survey of the
literature emphasizing the relationship between the sizes of the
validation set and the training set as well as discussing the positive
bias of cross-validation as a risk estimator.

Some results supporting the use of cross-validation for statistical methods other
than lasso are known.  For instance, \citet*{Stone1974,Stone1977}
outlines various conditions 
under which cross-validated methods can result in good predictions.  More recently, 
\citet*{Dudoit2005} find finite sample bounds for various
cross-validation procedures.   
These results do not address the lasso nor parameter spaces with
increasing dimensions, and furthermore, apply to choosing the best
estimator from a finite collection of
candidate estimators.
\citet*{LecueMitchell2012} provide oracle inequalities related to using
cross-validation with lasso, however, it treats the problem as aggregating a dictionary of
lasso estimators with different tuning parameters, and the results are
stated for fixed $p$ rather than the high-dimensional setting
investigated here. Most recently, \citet*{FlynnHurvich2013}
investigate numerous methods for tuning parameter selection in
penalized regression, but the theoretical results hold only when $p/n
\rightarrow 0$ and not for cross-validation. In particular, the authors state
``to our knowledge the asymptotic properties of [$K$]-fold CV have not
been fully established in the context of penalized regression''
\cite*[p.~1033]{FlynnHurvich2013}. 

Rather than cross-validation, one may use information criteria such as
AIC \citep{akaike1974new} or BIC \citep{Schwarz1978}. These techniques are frequently advocated in
the literature \citep[for example][]{BuhlmannGeer2011,
  WangLi2007,Tibshirani1996,FanLi2001}, but the classical
asymptotic arguments underlying these methods apply only for $p$ fixed and
rely on maximum likelihood estimates (or Bayesian posteriors) for all
parameters including the noise. The theory in the high-dimensional
setting supporting these methods is less complete. Recent work has
developed new information criteria with supporting asymptotic results
uf $\textrm{rank}(\mathbb{X}) = p$
but is still allowed to increase. For example, the criterion developed by \citet{wang2009shrinkage}
selects the correct model asymptotically even if $p\rightarrow
\infty$ as long as $p/n\rightarrow 0$. 
If $p$ is allowed to increase more quickly than $n$, theoretical
results assume $\sigma^2$ is known to get around the difficult task of
high-dimensional variance estimation 
\citep{chen2012extended,ZhangShen2010,kim2012consistent,FanTang2013}. 

%None
%of this existing theory, however, is general enough to cover the
%high-dimensional setting we investigate with $p/n\rightarrow\infty$
%and $\sigma^2$ unknown.

In \autoref{sec:notation-assumptions}, we outline the mathematical
setup for the lasso prediction problem and discuss some empirical
concerns. \autoref{sec:mainresults} contains the main result and
associated conditions. \autoref{sec:simulations} compares
different choices of $K$ for cross-validation via
simulation, while
\autoref{sec:conclusion} summarizes our contribution and presents some
avenues for further research.

% \autoref{sec:proofSection} presents some useful 
% lemmas and provides the proof of our results, w

\section{Notation and definitions}
\label{sec:notation-assumptions}

%In this section, we present a mathematical setup of the cross-validation problem.
\subsection{Preliminaries} 
Suppose we observe data $Z_{i}^{\top} = (Y_{i},X_{i}^{\top})$
consisting of predictor variables, $X_{i} \in \mathbb{R}^{p_{n}}$, and response
variables, $Y_{i}\in\mathbb{R}$, where $Z_{i} \sim\mu_n$ are
independent and identically distributed for $i
=1,2,\ldots,n$ and the distribution $\mu_n$ is in some class $\F$ to be
specified.  Here, we use the notation $p_n$ to allow the number
of predictor variables to change with $n$. Similarly, we index the
distribution $\mu_n$ to emphasize its dependence on $n$. For simplicity, 
in what follows, we omit the subscript $n$ when there is
little risk of confusion.

We consider the problem of estimating a linear functional
$
f(\Xrv) =  \Xrv^\top \beta
$
for predicting $\Yrv$,
when $\Zrv^{\top} = (\Yrv,\Xrv^{\top}) \sim \mu_n$ is a new, independent random variable from the same
distribution as the data and
$\beta=(\beta_1,\ldots,\beta_{p})^{\top}$.  For now, we do not assume
a linear model, only the usual regression framework where
% \begin{equation}
% \label{eq:regressionFunction}
$\Yrv = f^*(\Xrv) + \epsilon,$
%\end{equation}
with $\epsilon$ and $\Xrv$ independent and $f^*$ is some unknown function.
We will use zero-based
indexing for $\Zrv$ so that $\Zrv_0 = \Yrv$.
To measure performance, we use the $L^2$-risk of the predictor $\beta$:
\begin{equation}
  \label{eq:risk-main}
  \risk{\beta} := \E_{\mu_n} \left[ (\Yrv
    -\Xrv^\top \beta)^2\right].
\end{equation} 
Note that this is a conditional expectation: for any estimator $\hat\beta$,
\begin{equation}
  \label{eq:risk-main2}
  \risk{\hat\beta} := \E_{\mu_n} \left[ (\Yrv
    -\Xrv^\top \hat\beta)^2 | Z_1,\ldots,Z_n\right],
\end{equation} 
and the expectation is taken only over
the new data $\Zrv$ and not over any observables which may be used to choose $\hat\beta$.

Using the $n$ independent observations $Z_1,\ldots,Z_n$, we can form
the response vector $Y := (Y_i)_{i=1}^n$ and the design matrix 
$\mathbb{X} :=  [X_1,\ldots,X_n]^{\top}$.  Then, given a
vector $\beta$, we write the squared-error empirical risk function as
\begin{equation}
  \label{eq:empirical-risk-main} 
  \emprisk{\beta} := \frac{1}{n}|| Y - \mathbb{X}\beta||_2^2 = \frac{1}{n}
  \sum_{i=1}^n (Y_i- X_i^{\top}\beta)^2.
\end{equation} 
Analogously to \refeq{eq:empirical-risk-main}, we write the
\emph{$K$-fold cross-validation estimator of the risk with respect to
the tuning parameter $t$}, which we
abbreviate to CV-risk or just CV,  as
\begin{equation}
  \cvrisk{V_n}{\gen}   =
  \cvrisk{V_n}{\hat\beta_\gen^{(v_1)},\ldots,\hat\beta_\gen^{(v_{K})}}
  := \frac{1}{K} \sum_{v \in V_n}
  \frac{1}{|v|} \sum_{r \in v} \left(Y_r - X_r^{\top}\hat\beta_\gen^{(v)}
  \right)^2.
    \label{eq:cv-risk-main}
\end{equation}
Here, $V_n = \{ v_1 , \ldots, v_{K} \}$ is a set of
validation sets, $\hat\beta_\gen^{(v)}$ is the estimator in
\refeq{eq:optimization} with the observations in the validation
set $v$ removed, and $|v|$ indicates the
cardinality of $v$.
Notice in particular that the cross-validation estimator of the risk is
a function of $\gen$ rather than a single predictor $\beta$---it uses
$K$ different predictors at a fixed $t$, averaging over their
performance on the respective held-out sets. Over the course of
the paper, we will
freely exchange $\lambda$ for $t$ in this definition. 

Lastly, we define the CV-risk minimizing choice of tuning parameter to be
\begin{equation}
  \label{eq:cv-objective} 
  \hat{\gen} = \argmin_{\gen\in\Gen} \cvrisk{V_n}{\gen}.
\end{equation} 
In our setting, we will take $\Gen$ (or $\Lambda$) as an interval subset of the nonnegative
real numbers which needs to be defined by the data-analyst.  The
choice of $\Gen$ is an important part of the performance of
$\hat\beta_{\hgen}$ and requires some explanation. First, we provide some insight into the 
computational load of using CV-risk to find $\hat{\gen}$.

\subsection{Computations}
In practice, CV-risk is known to be time consuming and somewhat unstable due to
the randomness associated with forming $V_n$.  For a fixed $v \in V_n$, 
suppose $\hat\beta_\lambda^{(v)}$ is found for the entire lasso path via the {\tt lars}  \citep{EfronHastie2004}
algorithm, which can be computed in the same computational complexity as a least squares
fit.  To fix ideas, suppose $n > p$, which means {\tt lars} has computational complexity $O(np^2 + p^3)$.
Hence, as the feature matrix $\X$ with $|v|$ rows removed has approximately 
$n(K-1)/K$ rows, $\hat\beta_{\lambda}^{(v)}$ can be computed  for all $\lambda$ in $O((n(K-1)/K)p^2 + p^3)$ 
time.
Repeating this $K$ times means the computational complexity for forming
$\cvrisk{V_n}{\lambda}$ over all $\lambda$ is $O( n(K-1)p^2 + p^3)$.  If $K$ is a fixed
fraction of $n$, CV-risk has computational complexity of order $(np)^2$, which is a factor of $n$
slower than a single lasso fit.  

Though more expensive on a single processor, CV-risk is readily
parallelizable over the $K$ folds and therefore (ignoring communication costs between processors)
CV-risk could actually be {\it faster} to compute than $\hat{R}$ (and hence $\hat\beta_\lambda$) as 

\noindent$n(K-1)/K < n$.  This advantage
is of course lost if we ultimately compute $\hat\beta_{\hat\lambda}$.  However, this computational advantage 
would be maintained  if we instead report
\begin{equation}
  \label{eq:cv-modelAverage} 
\tilde{\beta} = \frac{1}{K}\sum_{v \in V_n} \tilde\beta^{(v)},
\end{equation} 
where $\tilde\beta^{(v)}$ is the lasso estimate trained on the observations in $(v)$ with 
the tuning parameter
chosen by minimizing the empirical risk using the test observations in $v$. For $K = 4$,
for example, this would provide a 25\% reduction in computation time.  The properties of 
this approximation is an interesting avenue for further investigation.

%We explore the affect of $K$ on prediction error and model selection via simulation
%in \autoref{sec:simulations}.

% It is convenient to also
% write
% \begin{equation}
%   \label{eq:quadFormRisk} \E_{\mu_n} \left[ (\Yrv -\Xrv\beta)^2\right]
%   := \gamma^{\top} \Sn \gamma,
% \end{equation} where $\Sigma_{n,jk} = \E_{\mu_n}[ \Zrv_j \Zrv_k]$,
% $0\leq j,k,\leq p$ and $\gamma = (-1,\beta_1,\ldots,\beta_p)^{\top}$.

\subsection{Choosing the sets $\Lambda$ and $\Gen$}
\label{sec:Tn} 

The data analyst must be able to solve the optimization
problem in \refeq{eq:cv-objective}. 
For $\Lambda$, we must choose a lower bound: $\Lambda =
[\lambda_n,\infty)$. This implies we
must choose $\lambda_n$ as a function of the data. While it is
tempting to allow $\lambda_n=0$, this results in numerical 
and practical implementation issues if rank$(\X) < p$ and is
unnecessary as the theory will show. However, the lower bound will have a
nontrivial impact on the quality of the recovery, as choosing a value
too large may eliminate the best solutions.  
%Thus, treating
%the lower bound as a random function of the data is more realistic
%from a statistical practice point of view. 
We will see in the next
section that under some assumptions on the data generating process, we
can safely choose a particular $\lambda_n>0$ that allows order $\log n$
coefficients to be selected without compromising the quality of the estimator.

In the case of $\Gen$, an upper bound must be selected for any
grid-search optimization procedure. As we will impose much weaker
conditions on the data generating process, choosing such an upper
bound is more complicated.
%. Note that more advanced
%optimization techniques are not practical as the CV-risk objective function
%in \refeq{eq:cv-objective} is, in general, non-convex. 
% To define such an upper bound in a practical way, it should be large
% enough to include all possible estimators in a given class while still
% being finite.  This implies we must choose $\Gen$ to be a function of
% the data. 
% The specifics of $\Gen$
%depend on the regularizing set $\B_\gen$. 
Note that, by Equation~\eqref{eq:optimization},
$\hat\beta_t$ must be in the $\ell_1$-ball with radius $t$.  This
constraint is only binding \citep*{OsbornePresnell2000} if
\begin{equation*}
  \label{eq:upperBoundTn} 
  t < \min_{\eta \in \mathcal{K}} || \ols + \eta||_1 =: t_0,
\end{equation*} 
where $\ols = \hat\beta(\mathbb{R}^p)$ is a least-squares solution and
$\mathcal{K} := \{a : \mathbb{X}a = 0\}$.  Observe that $\mathcal{K} =
\{0\}$ if $n \geq p$ and otherwise $\mathcal{K}$ has dimension $p - n$,
which would imply $\ols$ is not unique. Both of these
  statements assume that the columns 
  of $\mathbb{X}$ contain a linearly independent set of size
  $\min\{n,p\}$. See \citet*{Tibshirani2013} for the more general
  situation. In either case, if $t \geq 
t_0$, then $\betat{t}$ is equal to a least-squares solution.
Based on this argument and the fact that $0 \in \mathcal{K}$, it is tempting to
define the upper bound to be  $\tparm_{\max} := ||\ols||_1$, 
where $\ols = (\X^{\top}\X)^{\dagger}\X^{\top}Y$ is the least-squares solution when
$(\cdot)^{\dagger}$ is given by the Moore-Penrose inverse.  However, this upper bound
has at least two problems.  First, although theoretically
well-defined, as with setting $\lambda_n=0$, it suffers from numerical 
and practical implementation issues if rank$(\X) < p$. The second
problem is that it grows much too fast, at least order $\sqrt{n}$,
therefore potentially including solutions which will have low bias but
very high variance.

Instead, we define an upper bound based on a rudimentary variance estimator
$
\tparm_{\max} := \norm{Y}_2^2/a_n,$
 % \label{eq:tmax}
% \end{equation}
where $(a_n)$ is an increasing sequence of constants.  Hence, we
consider the optimization interval 
to be $T = [0,\tparm_{\max}]$.   We defer fixing a particular sequence
$(a_n)$ until the next section. As our main results demonstrate, this
choice of $T$ is enough to ensure that a risk consistent
sequence of tuning parameters can be selected.

\begin{remark}
  We emphasize here that using a cross-validated tuning parameter
  involves more than simply allowing the tuning parameter to be
  selected in a data-dependent manner. In
  order to meaningfully analyze tuning parameter selection, we allow
  the search set $T$ and the tuning parameter to be chosen based on
  the data. 
\end{remark}

\begin{remark} 
  The computational implementation of CV for an interval $\Lambda$ (or $T$)
  deserves some discussion. Two widely used algorithms for lasso
  are {\tt glmnet} \citep{FriedmanHastie2010}, which uses coordinate
  descent, and {\tt lars} \citep{EfronHastie2004}, which leverages the
  piece-wise linearity of  the lasso solution as
  $\lambda$ varies (the lasso path).  The package {\tt 
    glmnet} is much faster than 
  {\tt lars}, however, {\tt glmnet} examines a grid of values,
  $\lambda_j \in \Lambda$, $j=1,\ldots,J$ say,
  and approximates the solution at each $\lambda_j$ with increasing
  accuracy depending on the number of iterations.  On the other hand, {\tt
    lars} evaluates the entire solution path exactly, such that it is
  theoretically possible to choose any $\lambda\in\Lambda$ via
  numerical optimization. However, optimizing Equation
  \eqref{eq:cv-objective} with standard solvers 
  can be difficult due to a possible lack of convexity. In
  both cases, the extremes of the interval $\Lambda$ will affect the
  quality of the solution.
\end{remark}

\section{Main results}
\label{sec:mainresults} 

In this section, we present results for both forms of the lasso
estimator in Equations~\eqref{eq:lagrangian}
and~\eqref{eq:optimization} under more and less restrictive
assumptions, respectively. First, we define the types of random
variables $\Zrv$ which we allow. 
To quantify the tail behavior of our data, we appeal to the notion of an Orlicz norm.

\begin{definition}
\label{def:Orlicz}
For any random variable $\xi$ and
function $\psi$ that is nondecreasing, convex, and $\psi(0) = 0$,
define the $\psi$-Orlicz norm
\[
||\xi||_{\psi} = \inf \left\{ c > 0 :  \E \psi \left( \frac{|\xi|}{c} \right) \leq 1 \right\}.
\]
\end{definition}
For any integer $r \geq 1$, we are interested in both the usual
$L^r$-norm given by $||\xi||_r := (\E |\xi|^r)^{1/r}$ and the  
norm given by choosing $\psi(x) = \psi_r(x) := \exp(x^r) - 1$
and represented notationally as $||\xi||_{\psi_r}$. Note that if
$||\xi||_{\psi_r} < \infty$, then for sufficiently large $x$, there are
constants $C_1,c_2$ such that 
$
P(|\xi| > x) \leq C_1 \exp(-c_2 x^r).
$

% The converse also holds in the sense that if \refeq{eq:expTailBound}
% holds, then $\norm{\xi}_{\psi_r} < \infty$ 
% \citep[e.g.][]{vanweak}. 
In the particular case of the $\psi_2$-Orlicz
norm, if $||\xi||_{\psi_2} < \kappa$
it holds that
  $\P(|\xi|>x) \leq 2 \exp(-x^2/\kappa^2)$
 and  $\E[|\xi|^{k}] \leq 2 \kappa^{k}\Gamma(k/2+1),$
where $\Gamma(t) = \int_0^\infty x^{t-1}e^{-x}dx$ is the Gamma
function. Additionally,  $(\E | \xi|^r)^{1/r} = || \xi||_r \leq r! || \xi||_{\psi_1}$ and
  $|| \xi||_{\psi_1} \leq (\log 2)^{1/r-1 } || \xi||_{\psi_r}$.   

% A random variable is said to be sub-Gaussian if and only if
% it has finite $\psi_2$-Orlicz norm. 
Before outlining our results, we define the following set of distributions.  
\begin{definition}
  \label{cond:Orlicz} 
Let $1 \leq q < \infty$ and $\Cf$ be a constant independent of $n$. Then define
  \[
  \F_q := \left\{ (\mu_n):  \mu_n \textrm{ is a measure on } \R^{p_n} \textrm{ and } \max_{1 \leq j,k \leq p} || \xi_j\xi_k -
    \mathbb{E}_{\mu_n} \xi_j\xi_k||_{\psi_q} \leq \Cf \right\}; 
  \]
  %This is right as we just need to bound an expectation (for a fixed n and p) with a constant independent of n
  that is, all centered cross products of a random variable $\xi\sim
  \mu_n$ have  uniformly finite $\psi_q$-Orlicz norm,  
  independent of $n$.
  Define the analogous set $\F_\infty$ which contains the measures $\mu_n$ such that  
  $|\xi_j| \leq \Cf$ almost surely $\mu_n$ for each
  $j=1,\ldots,p$. 
\end{definition} 

\begin{remark}
To make subsequent expressions as a function of $q$ make sense,
interpret for any $0< c < \infty$,  
$c/\infty = 0$ and $\infty/\infty = 1$.
\end{remark}

While $\mu_n$ is a measure on $\mathbb{R}^{p}$ indexing with $n$ is
more natural than indexing with $p$ 
given that our results include $p_n$ increasing with $n$.
\autoref{cond:Orlicz} is a common moment condition
\citep*{GreenshteinRitov2004,NardiRinaldo2008a,bartlett2012} for showing risk
consistency of lasso-type methods in high dimensional settings.
%We refer to this upper bounding constant for a given
%sequence of distributions as $\Cf$.  

% \begin{remark}
%   Something relating $\F_q$, Orlicz of $Y$ and the regression function.
% \end{remark}

% \begin{condition}
%   \label{cond:noise}
%   The noise, $\epsilon$, satisfies
%   $\norm{\epsilon}_{\psi_2} \leq \kappa < \infty$.
% \end{condition} 

% \begin{condition}
%   \label{cond:regressionFunction}
%   The regression function $f^*(X)$ satisfies
%   $\norm{f^*(X)}_{\psi_2}\leq F<\infty$.
% \end{condition} 

To simplify our exposition, we make the following condition about the size of the validation sets for CV.
\begin{condition}
  \label{cond:validation}
  The
   sequence of validation sets $\{V_n\}_{n=1}^\infty$ is such
  that, as $n\rightarrow\infty$, $\forall v \in V_n$, $|v| \asymp c_n$
  for some sequence $c_n$.  
  % Additionally, for each $n$, if $v, v' \in
  % V_n$ and $v \neq v'$, then  $v \cap v' = \emptyset$.
\end{condition}

This condition is intended to rule out some pathological cases of
unbalanced validation sets, but standard CV methods satisfy it. For
example, with $K$-fold cross-validation, 
we can take $c_n = \lfloor n/K \rfloor$, which is the integer
part of $n/K$.
For $n$ design random variables $X_1,\ldots,X_n$ and oracle prediction
function $f^*$, 
define $\mathbf{f}_n^* := (f^*(X_1),\ldots,f^*(X_n))^{\top}$. 
% Also,
% let $\esssup |f^*| := F^{1/2} < \infty$. 

\subsection{Persistence when $f^*$ is linear}
\label{sec:linear-model}

If we are willing to impose strong conditions on $\mu_n$, as in
~\citet*{BuneaTsybakov2007} and~\citet*{Meinshausen2007}, then we can
obtain cross-validated lasso estimators which achieve nearly oracle
rates. 

If we assume that $\E[\Yrv\ |\ \Xrv]=f^*(\Xrv) = \Xrv^\top \beta^* $, then we can write
the risk of an estimator $\hat{\beta}_{\hat{\lambda}}$ as
\[
R\left(\hat{\beta}_{\hat{\lambda}}\right) =
\underbrace{R\left(\hat{\beta}_{\hat{\lambda}}\right)- R(\beta^*)}_{\mbox{excess
    risk}} + \underbrace{R(\beta^*)}_{\mbox{noise}} =
\underbrace{R\left(\hat{\beta}_{\hat{\lambda}}\right)
  -\sigma^2}_{\mbox{excess risk}} + \underbrace{\sigma^2}_{\mbox{noise}},
\]
where $R(\beta^*)=\E[(\Yrv-\Xrv^\top\beta^*)^2] = \sigma^2$.
%\[
%R\left(\hat{\beta}_{\hat{\lambda}}\right) =
%\underbrace{R\left(\hat{\beta}_{\hat{\lambda}}\right)- R(\beta^*)}_{\mbox{excess
%    risk}} + \underbrace{R(\beta^*)}_{\mbox{noise}} =
%\underbrace{R\left(\hat{\beta}_{\hat{\lambda}}\right)
%  -\Xrv\beta^*}_{\mbox{excess risk}} + \underbrace{\sigma^2}_{\mbox{noise}}.
%\]
We write the \emph{excess risk} as
$
\mathcal{E}(\hat{\lambda}) := R(\hat{\beta}_{\hat{\lambda}})- \sigma^2
$
and prove a convergence rate for $\mathcal{E}(\hat{\lambda})$. In this
case, targeting the excess risk is the same as estimating the
conditional expectation of $\Yrv$, but if
$f^*(\Xrv)$ is not linear (as in \autoref{sec:nonlinear-model}), the excess risk remains a
meaningful way of assessing prediction behavior.

\begin{theorem}
  \label{thm:simpleCV}
   %Following \citet*{Meinshausen2007,BuneaTsybakov2007}, 
  Assume the following conditions:
 % \begin{itemize}
 % \item[M1:] $\norm{\Xrv}_\infty < \Cf$;
 % \item $\E[\Xrv]=0$ and $\E[\Xrv \Xrv^\top] = \Sigma$;
 % \item $\epsilon\sim N(0,\tau)$;
 % \item $\exists 0 < \nu <1$ such that $\Sigma$ and $\Sigma^{-1}$ are
 %   diagonally dominant at level $\nu$; that is $|\sigma_{jj}|
 %   \geq \nu + \sum_{j\neq i}|\sigma_{ij}|$ for all $1\leq j\leq
 %   p$. Note that this implies that
 %   $\Sigma-(1-\nu)\textrm{diag}(\Sigma)$ is positive semi-definite.
 % \item $\norm{\beta^*}_0=s^*$, which is independent of $n$.
 % \end{itemize}
  \begin{compactenum}[\bf {M}1:]
    \item There exists a constant $\Cf<\infty$ independent of $n$ such that $|\Xrv_j| < \Cf$ almost surely for all $j = 1,\ldots,p$. 
    \item $\E[\Xrv]=0$ and $\E[\Xrv \Xrv^\top] = \Sigma$.
    \item  $\epsilon\sim N(0,\sigma^2)$
    \item $\exists 0 < \nu <1$ such that $\Sigma$ and $\Sigma^{-1}$ are
      diagonally dominant at level $\nu$; that is 
      $|\sigma_{jj}| \geq \nu + \sum_{j\neq i}|\sigma_{ij}|$ for all $1\leq j\leq
      p$. 
    \item $||\beta^*||_0=s^*$, which is independent of $n$.
    \item $\lambda_{\min} \propto (\log{n}\log{p}/n)^{1/2}$.
    \item $\log{p} = o(n/\log{n})$.
    \end{compactenum}
    \vspace{.25in}
    
    \noindent Then,
    $
    \mathcal{E}(\hat{\lambda})  = O_p\left((s^*\log n \log p)/n\right).
    $
\end{theorem}
These conditions and the result warrant a few comments.
First, note that condition {\bf M4} implies that
$\Sigma-(1-\nu)\textrm{diag}(\Sigma)$ is positive
semi-definite. Second, as $s^*$ is
fixed, {\bf M6} and {\bf M7} ensure $\lambda_{\min}\rightarrow 0$
so that  $\Lambda$ will eventually allow models with $s^*$
covariates. Thus, the procedure is asymptotically consistent for model
selection \citep{Meinshausen2007}. 
% Third, we have that, under the 
% conditions of \autoref{thm:simpleCV}, $R(\beta^*)=\sigma^2$.  
Finally, note
that $ \mathcal{E}(\hat{\lambda})$ is random due to the term
$R(\hat\beta_{\hat{\lambda}})$.  Here, we emphasize that
$R(\hat\beta_{\hat{\lambda}})$ is a function 
of the data: the conditional expectation of
a new test random variable $\Zrv$ given the observed
data used to choose both $\hat{\lambda}$ and $\hat\beta_{\lambda}$ as
in \refeq{eq:risk-main2}. 

While the conditions of \autoref{thm:simpleCV} are strong, they are
typical of those used to prove persistence of the lasso estimator with
oracle tuning parameter
\citep[for the case of fixed design, see][]{NegahbanRavikumar2012}. For instance,
\citet{BuneaTsybakov2007} prove an oracle rate for the lasso of
$O(s^*\log{p}/n)$ with $\lambda_{\min} \propto \sigma\sqrt{\log{p}/n}$. 
Under similar conditions, our result with
cross-validated tuning parameter requires a larger $\lambda_{\min}$
(resulting in smaller models) and a slower convergence rate to the
oracle by a factor $\log{n}$. A reasonable choice of $\Lambda$
suggested by \autoref{thm:simpleCV} is
$\Lambda=[\lambda_{\min},\ \infty) = [(\log{p}\log{n}/n)^{1/2},\
  \infty)$.

\begin{proof}[Proof of \autoref{thm:simpleCV}]
  We have that, for all $g>0$, 
  \begin{align*}
    %\lefteqn{
    \P\left(R(\hat{\beta}_{\hat{\lambda}})-\sigma^2 > g\frac{s^*\log n \log
    p}{n}\right)%}\\
 &\leq \P\left(\inf_{\lambda \in \Lambda}\left(
   R(\hat{\beta}_{\lambda})-\sigma^2\right) > g\frac{s^*\log n
   \log 
   p}{2n}\right)\\
 &+ 2 \P\left(\sup_{\lambda\in\Lambda}\left|
   R(\hat{\beta}_{\lambda})-\sigma^2-\hat{R}(\hat{\beta}_\lambda)\right| > g\frac{s^*\log n
   \log p}{2n}\right),
  \end{align*}
  by the proof of Theorem 7 in \cite*{Meinshausen2007}. Note that
  $\hat{R}(\hat{\beta}_\lambda)$ is defined in  \refeq{eq:empirical-risk-main}.
Furthermore,
  the second term on the right hand side goes to 0 by that result. 
  Now note that
  \begin{align*}
    \P\left(\inf_{\lambda \in \Lambda}\left(
    R(\hat{\beta}_{\lambda})-\sigma^2\right) > g\frac{s^*\log n
    \log 
    p}{n}\right) &\leq \P\left(
                   R(\hat{\beta}_{\lambda_{\min}})-\sigma^2 > g\frac{s^*\log n
                   \log 
                   p}{n}\right).
  \end{align*}
  By 
  Corollary 1 in \cite*{BuneaTsybakov2007}, for any $\lambda$, we have
  that
  \begin{align*}
    \P\left( R(\hat{\beta}_\lambda)-\sigma^2 > g\frac{s^* \lambda^2}{1-\nu}\right)
    \leq 10 p^2 \exp\left( -c_1n\lambda^2\right) = 10 \exp\left(
    2\log p - c_1n\lambda^2\right).
  \end{align*}
  Setting $\lambda_{\min}$
  proportional to 
  $(\log n\log p/n)^{1/2}$  is enough for the upper bound to go to
  zero as $n\rightarrow\infty$ yielding the result.
\end{proof}

\subsection{Persistence when $f^*$ is not linear}
\label{sec:nonlinear-model}

In order to derive results under more general conditions, we move to
the linear oracle estimation framework.
For any $\gen$, define the oracle estimator with respect to $\gen$ as
\begin{equation*}
  \label{eq:estimator-oracle}
  \beta_{\gen} := \argmin_{\beta \in \B_{\gen}} \risk{\beta}.
\end{equation*} 
Suppose $\hgen$ is an estimator, such as by
cross-validation.  Then we can decompose the risk of an estimator
$\hat{\beta}_{\hat{\gen}}$ as follows:
\begin{equation*}
  \label{eq:riskdecomp}
  \risk{\hat{\beta}_\gen} = \underbrace{\risk{\hat{\beta}_{\hat{\gen}}} -
    \risk{\beta_\gen}}_{\mbox{excess risk}}
  \quad + \underbrace{\risk{\beta_{\gen}}-R_*}_{\mbox{approximation
      error}} + \quad \underbrace{R_*}_{\mbox{noise}},
\end{equation*}
where $R_*$ is the risk of the mean function $f^*$.
Because the data generating process is not necessarily linear, we
study the performance of an estimator $\hat\beta_{\hgen}$ via the 
{\it excess risk of} $\hat\beta_{\hgen}$ {\it relative to}
$\beta_{\gen}$, which we define as 
\begin{equation}
  \label{eq:excess-risk-general} 
  %\mathcal{E}(\Sest,\Sor) 
   \mathcal{E}(\hgen,\gen) 
  := \risk{\hat\beta_{\hgen}} - \risk{\beta_{\gen}}.
\end{equation} 
Here, $ \mathcal{E}(\hgen,\gen) $ depends on the cross-validated
tuning parameter $\hgen$ as well as the oracle set through
$\gen$. Focusing on \refeq{eq:excess-risk-general}, allows for meaningful
theory when the approximation error does not
necessarily converge to zero as $n$ grows. 
This is particularly important here,
as we do not assume that the conditional expectation of $\mathcal{Y}$
given $\mathcal{X}$ 
is linear. 
As $t \rightarrow \infty$, the approximation error decreases and hence
 we desire to take $t=t_n$ for some
increasing sequence $(t_n)$.
% Note that $ \mathcal{E}(\hgen,\gen) $
% is random due to the term $\risk{\hat\beta_{\hgen}}$.  Here, we
% emphasize that
% $\risk{\hat\beta_{\hgen}}$ is a function 
% of the data: the conditional expectation of
% a new test random variable $\Zrv$ given the observed
% data used to choose both $\hgen$ or $\hat\beta_{\gen}$ as in \refeq{eq:risk-main2}. 
As discussed in the introduction, \citet*{GreenshteinRitov2004} show
that if $t_n=o((n/\log p)^{1/4})$, then $\mathcal{E}(t_n, t_n)$
converges in probability to zero. \citet*{bartlett2012}
increase the size of this search set allowing $t_n=o(n^{1/2}/(\log^{3/2} n
    \log^{3/2} (np)))$ while still having
$\mathcal{E}(t_n,t_n)\xrightarrow{P} 0.$

\begin{theorem}
  \label{thm:lasso} 
  Let $(\mu_n) \in \F_q$ and suppose that \autoref{cond:validation} holds. 
  Then for any sequences $(a_n),\ (t_n)$ which satisfy $a_nt_n
  = o(n)$,
  \begin{equation}
    \label{eq:theLasso} 
    \P_{\mu_n} \left(  \mathcal{E}(\tCV,t_n) > \delta \right) 
    \leq
    \delta^{-1} \left(\Omega_{n,1} + \Omega_{n,2} \right) + 2\P(D_n^c) + \P(E_n^c).
  \end{equation}
  Here,
  % \begin{align*}
  %   \Omega_{n,1} & := \left[1 + \frac{A_n + n + \sqrt{A_n n}}{a_n}\right]^2 2(\log p)^{1/2}\left(n^{-1/2} + c_n^{-1/2} + (n - c_n)^{-1/2}/2\right) \\
  %   \Omega_{n,2} & := (1+ t_n)^2 2\sqrt{\frac{\log p}{n}}
  % \end{align*}
  % and  $A_n =  \norm{\mathbf{f}_n^*}_2^2$.
  \begin{align*}
    \Omega_{n,1} & := \left[1 + \frac{2n \Cf'}{a_n}\right]^2 
    \sqrt{(\log p)^{1 + 2/q}}\left(n^{-1/2} + c_n^{-1/2} + (n - c_n)^{-1/2}\right), \\
    \Omega_{n,2} & := (1+ t_n)^2 \sqrt{\frac{(\log p)^{1 +
                   2/q}}{n}},\\
    D_n &:= \left\{ \tparm_{\max}  \leq \frac{2n\Cf'}{a_n} \right\},\\
    E_n &:= \{ \tparm_{\max} \geq t_n  \}
  \end{align*}
  % and  $A_n =  \norm{\mathbf{f}_n^*}_2^2$.
and $\Cf'=\Cf(\log 2)^{1/q-1}$. 
\end{theorem}

\begin{remark}
The sets $D_n,\ E_n$ account for cases wherein $t_{\max}=\norm{Y}_2/a_n$
results in suboptimal sets $T$. If $(a_n)$ is such that $t_{\max}$
grows too quickly with non-negligible probability, then
cross-validation may result in low-bias but high variance
solutions. On the other hand, if $t_{\max}$ is too small, then we will
rule out possible estimators with lower risk. 
Here $D_n$ calibrates a high-probability upper bound on $t_{\max}$
based on $(\mu_n)$ and the choice of $(a_n)$ while $E_n$ ensures that
$t_{\max}$ will be large enough to include low risk estimators. 
\end{remark}
\begin{remark} 
  Usually in the oracle estimation framework, $\tCV = t_n$ and so the
  excess risk is necessarily nonnegative because the oracle predictor,
  $\beta_{t_n}$, is selected as the 
  risk minimizer over $\B_{t_n}$.  In that case, \refeq{eq:theLasso} would 
  correspond to convergence in probability.
  As we are
  examining the case where the optimization set is estimated,
  $\mathcal{E}(\tCV,t_n)$ may be negative.  However, we are only
  interested in bounding the case where $\mathcal{E}(\tCV,t_n)$ is positive, 
  i.e. the case where the estimator
  is worse than the oracle.
\end{remark} 

Define $b_n = \min\{n-c_n,c_n\}$. It follows that
$
(n^{-1/2} + c_n^{-1/2} + (n - c_n)^{-1/2}) \leq 3b_n^{-1/2}.
$
We state the following corollary to \autoref{thm:lasso}, which outlines the
conditions under which 
we can do asymptotically at least 
as well at predicting a new observation as the oracle linear model.  That is, when the
right hand side of \refeq{eq:theLasso} goes to zero.
\begin{corollary}
  \label{cor:lasso} 
 Under the conditions of  \autoref{thm:lasso}, if
 $a_n =n(\log p)^{1/4+1/(2q)} m_n/b_n^{1/4}$
 and 
 $t_n =o(b_n^{1/4}/m_n(\log p)^{1/4 + 1/(2q)})$, where $m_n$ is any
  sequence which tends toward infinity and $m_n = o(b_n^{1/4})$, we have that
for $n$ large enough and $p=o(\exp\{b_n^{q/(q+2)}\})$, 
%there exists a constant $\cConst$ independent of $n$ and $p$ such that
%\attn{changed from $p=o\left(e^{b_n^{q/(q+2)}}\right)$}
\[
 \P_{\mu_n} \left(  \mathcal{E}(\tCV,t_n) > \delta \right) 
        \leq
\frac{1}{m_n^2 \delta} \left( 1+\sqrt{\frac{b_n}{n}} \right)  
+2\exp(-n\cConstValFrac).
\]
In particular, $ \P_{\mu_n} \left(  \mathcal{E}(\tCV,t_n) > \delta \right)  \rightarrow 0$.
\end{corollary}

\begin{remark}
The rate at which $\delta$ can be taken to zero quantifies the decay of the size of the `bad' set where 
$\mathcal{E}(\tCV,t_n)$ is large. For the corollary, both $m_n = o(b_n^{1/4})$ and $\delta^{-1} = o(m_n^{2})$.
Therefore, it is necessary for $\delta^{-1} = o(b_n^{1/2})$ and hence, $\delta$ 
must go to zero at a slower rate than $b_n^{-1/2}$.  
\end{remark}
\begin{remark}
As $q$ increases, which corresponds to $(\mu_n) \in \F_q$ putting less
mass on the tails of the centered interactions 
of the components of $\Zrv$, the faster the oracle set given by
$\B_{t_n}$ can grow.  That, is the weaker the  $\ell_1$-sparsity 
assumption of the linear oracle. When $q = \infty$, the
random variables have no tails and we get 
the fastest rate of growth for $t_n$; that is, $(b_n^{1/4}/(m_n(\log p)^{1/4}))$.
\end{remark}

The parameter $b_n$ controls the minimum size of the validation versus
training sets that comprise cross-validation. 
It must be true that $b_n$ is strictly less than $n$. To get the best
guarantee, $b_n$ should increase as fast as possible. 
Hence, our results advocate a cross-validation scheme where the
validation and training sets are asymptotically balanced, i.e.\
$c_n\asymp n-c_n$.  
This should be compared with the results in \citet*{Shao1993}, which
state that for model selection consistency, 
one should have
$c_n/n\rightarrow 1$. However, \citet*{Shao1993} presents results for
model selection while we focus on prediction error. 
Similarly, \citet*{Dudoit2005} provide oracle inequalities for
cross-validation and also advocate for the validation set to grow as fast
as possible. Note that for $K$-fold cross-validation, $c_n=\lfloor
n/K \rfloor$ so that $b_n=O(n)$.

Additionally, the rates $a_n$ and $m_n$ deserve comment.  It is
instructive to compare this 
choice of $\tmax$ with $\norm{Y}_2^2/n$, a standard estimate of the
noise variance in high dimensions \citep*[e.g.][p.\ 104]{VanDeGeer2011}.
If $a_n = n$, then $\norm{Y}_2^2/a_n$ is an overestimate of the variance due to the presence of the 
regression function $f^*$. Our results state that we must choose $a_n$
to increase slower than $n$, thereby increasing this overestimation
and enlarging the potential search set $T$.
While $a_n$ depends on several parameters, $b_n$, $n$, and $p$ are
known to the analyst.  Also, the choice 
of $q$ depends on how much approximation error the analyst is willing
to make. 
The $m_n$ term is required to slow the growth of $t_n$ just
slightly. While this shrinks the size of the set
$\B_{t_n}$ relative to that used by \citet*{GreenshteinRitov2004}, potentially
eliminating some better solutions, effectively $m_n$ quantifies their
requirement $t_n=o((n/\log p)^{1/4})$, making explicit the need for
$t_n$ to grow more slowly than $(n/\log p)^{1/4}$.
As such, if we set $b_n \asymp n$ and set $q = \infty$, 
we reaquire the rate shown by
\citet*{GreenshteinRitov2004}, where \autoref{cor:lasso}  
implies that both 
$ \P_{\mu_n} \left(  \mathcal{E}(\tCV,t_n) > \delta \right) \rightarrow 0$
and $t_n = o( (n/\log p)^{1/4} )$. 

While the entire proof is contained in the supplementary material, we provide a brief sketch here.
\begin{proof}[Proof sketch of \autoref{thm:lasso} and \autoref{cor:lasso}]
  In order to control the behavior of $\P_{\mu_n} \left(
    \mathcal{E}(\tCV,t_n) > \delta \right)$, we require a number of steps.
  \begin{compactenum}
  \item The first step is form the decomposition
    \begin{align*}
      \mathcal{E}( \hgen,\ngen ) 
      &= 
        \risk{\hat\beta_{\hgen}} - \risk{\beta_{\ngen}}
      = 
        \underbrace{\risk{\betat{\hgen}} -\cvrisk{V_n}{ \hgen}}_{(I)}+
      \\ &+
        \underbrace{\cvrisk{V_n}{\hgen} -  \cvrisk{V_n}{\mgen}}_{(II)} 
        + \underbrace{\cvrisk{V_n}{\mgen} - \emprisk{\betat{\ngen}}}_{(III)}
        + \underbrace{\emprisk{\betat{\ngen}} -
        \risk{\beta_{\ngen}}}_{(IV)}.
    \end{align*}
    Here, 
     $\emprisk{\betat{\ngen}}$ and 
     $ \cvrisk{V_n}{\hgen}$ are defined in 
\refeq{eq:empirical-risk-main} and
      \refeq{eq:cv-risk-main}, respectively.
  \item By standard rules of probability, it is enough to control
    each term (I) through (IV) on the intersection $D_n
    \cap E_n$ as long as $\P(D_n^c)$ and $\P(E_n^c)$ are both small.
  \item For each of the terms, we rewrite the
    difference in risks as a difference of quadratic forms involving the
    covariance matrices of the random variables and the extended
    coefficients. For example, 
    $
    \risk{\beta} = \E_{\mu_n}  (\Yrv -\beta^\top\Xrv)^2=
    \gamma^{\top} \Sn \gamma 
    $
    where $\Sigma_n = \E_{\mu_n}(\Zrv \Zrv^{\top})$ and $\gamma^{\top}=(-1,\beta^{\top})$,
    and
    $    
    \emprisk{\beta} = n^{-1} \norm{Y- \mathbb{X}\beta}_2^2 =
    \gamma^{\top} \Sign \gamma, 
    $
    where $\Sign = n^{-1}\sum_{i=1}^n Z_{i}Z_{i}^{\top}$. Thus, we
    have, for example,
    $
    (IV) = \hat{\gamma}^\top_{t_n} \Sign \hat{\gamma}_{t_n} -
    \gamma^\top_{t_n} \Sn \gamma_{t_n}.
    $
  \item After some algebraic manipulation and an application of
    H\"{o}lder's inequality, we can rewrite these
    differences as the product of the squared $L^1$-norm of the coefficient
    vector and the expected value of the $L^\infty$-norm of the
    difference between an empirical and expected covariance. For
    example
    $
    \risk{\beta} - \emprisk{\beta} \leq ||\gamma||^2_1||\Sn-\Sign||_\infty.
    $
  \item Finally, the $L^1$-norm of the coefficient vector is bounded
    by $t$ due to the constraint in Equation~\eqref{eq:optimization}
    while $\E||\Sn-\Sign||_\infty=O((\log (p)/n)^{1/2})$ by
    Nemirovski's inequality and some intermediate lemmas.
  \item To move to the corollary, we use the Orlicz condition on $(\mu_n)$ to show that
    $\P(D_n^c)$ and $\P(E_n^c)$ are both small with high
    probability. This ensures that 
    $t_{\max}<2n\Cf'/a_n$ and for any sequences
    $(a_n)$ and $(t_n)$ which satisfy $a_nt_n=o(n)$, $t_{\max}>t_n$. 
  \end{compactenum}
\end{proof}

We note here that our results generalize to other $M$-estimators which
use an $\ell_1$-constraint. In particular, relative to
the set of coefficients $\beta \in \B_{t_n}$ with $t_n=o\left(b_n^{1/4}/\left(m_n
(\log p) ^{1/4+1/(2q)}\right)\right)$, an
empirical estimator with cross-validated tuning parameter has
prediction risk which converges to the prediction risk of the oracle.
\begin{corollary}
  \label{cor:lassolike}
  Consider the group lasso \citep*{YuanLin2005}
  \[
  \hat{\beta_t} = \argmin\{n^{-1}\norm{Y-\X\beta}_2^2 :
    \sum_{g\in G}\sqrt{|g|}\norm{\beta_g}_2 \leq 
    t\}, 
  \]
  and the square-root lasso \citep*{BelloniChernozhukov2014}
  \[
  \hat{\beta_t} = \argmin\{ n^{-1} \norm{Y-\X\beta}_2 :
    \norm{\beta}_1 \leq t\}. 
  \]
  Let
  $t_n$ and $a_n$ be as in \autoref{cor:lasso}. Then,
  for $n$ large enough, $\log(p)=o(b_n^{q/(q+2)})$, and $\max_g \sqrt{|g|}=O(1)$, we have
  \[
  \P_{\mu_n} \left(  \mathcal{E}(\tCV,t_n) > \delta \right) 
  \leq
  \frac{1}{m_n^2 \delta} \left( 1+\sqrt{\frac{b_n}{n}} \right)  
  + 2e^{-n\cConstValFrac}.
  \]
\end{corollary}

\section{Simulations}
\label{sec:simulations}

So far, though we have given some indication about the asymptotic order of $K$,
we haven't provided any guidance for the choice of $K$ with a fixed $n$ and $p$. 
In this section, we investigate the predictive and model selection  performance of $K$-fold
cross-validation for a range of $K$. 

We consider three criteria: prediction risk, sensitivity, and specificity.  For prediction risk, 
we approximate $R(\hat\beta_{\hat\lambda})$ by the empirical risk of 500 test observations.
For sensitivity and specificity, define the active set of a coefficient vector $\beta$ to be 
$\mathcal{S}(\beta) := \{j: |\beta_j| > 0\}$, along with 
$\mathcal{S}^* := \mathcal{S}(\beta^*)$ and  $\hat{\mathcal{S}} :=  \mathcal{S}(\hat\beta_{\hat\lambda})$
to be the active sets of $\beta^*$ and  $\hat\beta_{\hat\lambda}$,
respectively.  Then sensitivity = 
$|\mathcal{S}^* \cap \hat{\mathcal{S}}|/|\mathcal{S}^*|$ and 
specificity = 
$|(\mathcal{S}^*)^c \cap \hat{\mathcal{S}}^c|/|(\mathcal{S}^*)^c|$, where $|\cdot|$ counts the number of elements in
a set and $\mathcal{A}^c$ indicates the complement of a set $\mathcal{A}$.

\subsection{Simulation details}

\paragraph{Conditions:} For these simulations, we consider a wide range of possible conditions by varying the correlation in the design, $\rho$; the
number of parameters, $p$; the sparsity, $\alpha$; and the signal-to-noise ratio, SNR.  In all cases, we let $n = 100$, 
$p = 75, 350, 1000$, and set the
measurement error variance $\sigma^2 = 1$.  Lastly, we run each simulation condition combination 100 times.  For
these simulations, we assume that there exists a $\beta^*$ such that the regression function 
$f^*(X) = X^{\top}\beta^*$ in order to make model selection meaningful.

The design matrices are produced in two steps.  First, $X_{ij} \stackrel{i.i.d.}{\sim} N(0,1)$ for $1 \leq i \leq n$ and $1 \leq j \leq p$,
forming the matrix $\X \in \mathbb{R}^{n\times p}$.  Second, correlations are introduced by defining a matrix D that
has all off diagonal elements equal to $\rho$ and diagonal elements equal to one.  Then, we redefine $\X \leftarrow \X D^{1/2}$.
For these simulations, we consider correlations $\rho = 0.2, 0.5, 0.95$.

For sparsity, we take $s^* = \lceil n^{\alpha} \rceil$ and generate the $s^*$ non-zero elements of $\beta^*$ from a Laplace distribution
with parameter 1.  We let $\alpha = 0.1,0.33,0.5$.
Also, to compensate for the random amount of signal in each observation, we vary the signal-to-noise ratio,
defined to be 
$%\begin{equation*}
  \textrm{SNR} = \beta^{\top} D \beta.$
%\end{equation*}
We consider $\textrm{SNR} = 1$ and $10$. Note that 
as  the $\textrm{SNR}$ increases the observations go from a high-noise and low-signal regime to a low-noise and high-signal one.
Lastly, we consider two different noise distributions, $\epsilon \sim N(0,1)$ and $\epsilon \sim 3^{-1/2} t(3)$. 
Here $t(3)$ indicates a $t$ distribution with 3 degrees of freedom and
the $3^{-1/2}$ term makes the variance equal to 1.  Finally, we take
$K=\{3,\ 10,\ 25,\ 50,\ 75,\ 100\}$, the last case being leave-one-out CV.

\subsection{Simulation results}
Though we have run all of the above simulations, to save space we have only included plots from a select number of simulation conditions that are the most informative.  For prediction risk, all considered $K$'s result in remarkably 
similar prediction risks.  In \autoref{fig:predRisk}, for $p=1000$, SNR = 1, and $\rho = 0.9$, we see that taking
$K = 3$ or $K = 10$ results in slightly smaller prediction risks.
This is comforting as both of these values of $K$ 
are used frequently by default and they are the least computationally demanding.
\begin{figure} %This is SNR by Rho, with rho omitting 0.8
\centering
\begin{subfigure}[h]{.32\textwidth}  
	\includegraphics[width=1.5in,trim = 0 15 28 55,clip]
	{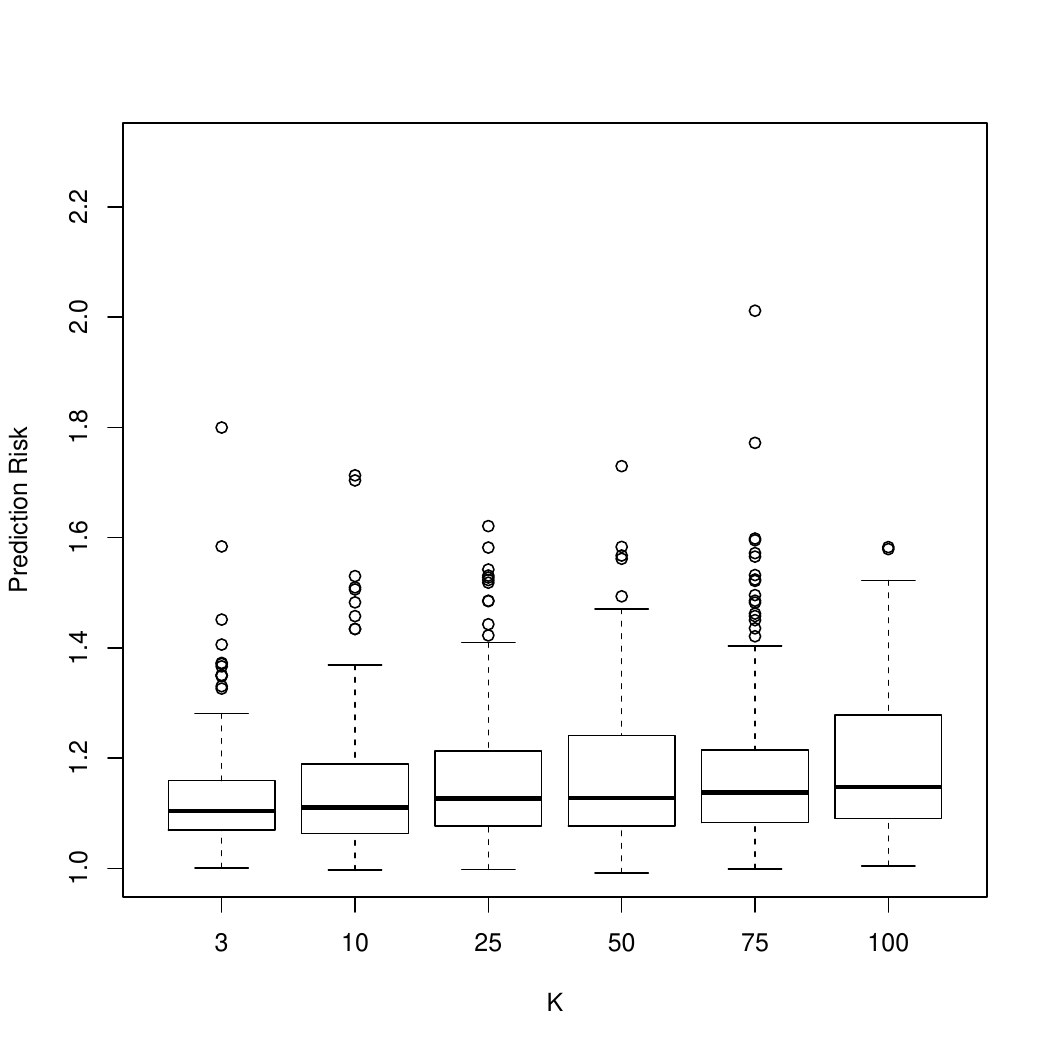}
	\caption*{$p = 1000$, $\alpha = 0.1$}
\end{subfigure}
\begin{subfigure}[h]{.32\textwidth}  
	\includegraphics[width=1.5in,trim = 0 15 28 55,clip]
	{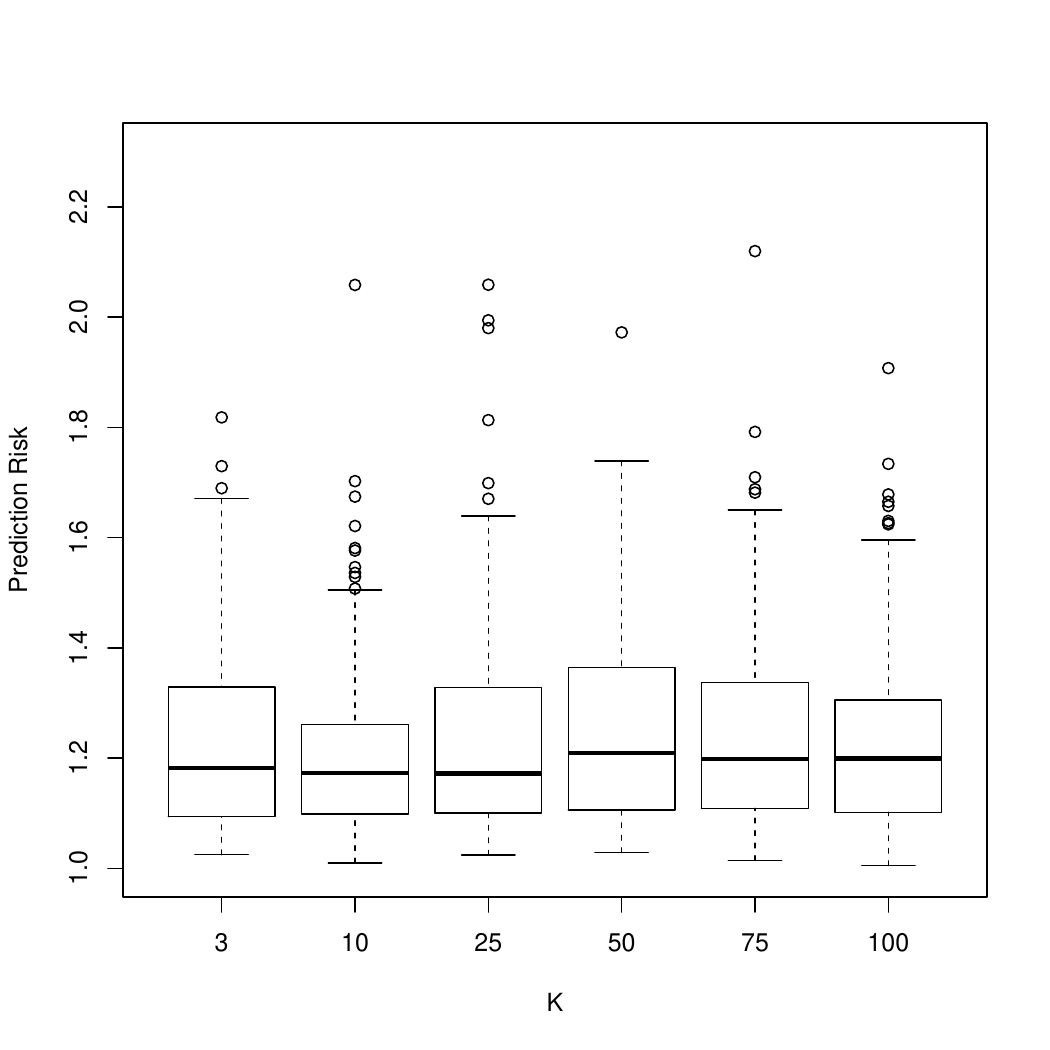}
	\caption*{$p = 1000$, $\alpha = 0.4$}
\end{subfigure}
\begin{subfigure}[h]{.32\textwidth}  
	\includegraphics[width=1.5in,trim = 0 15 28 55,clip]
	{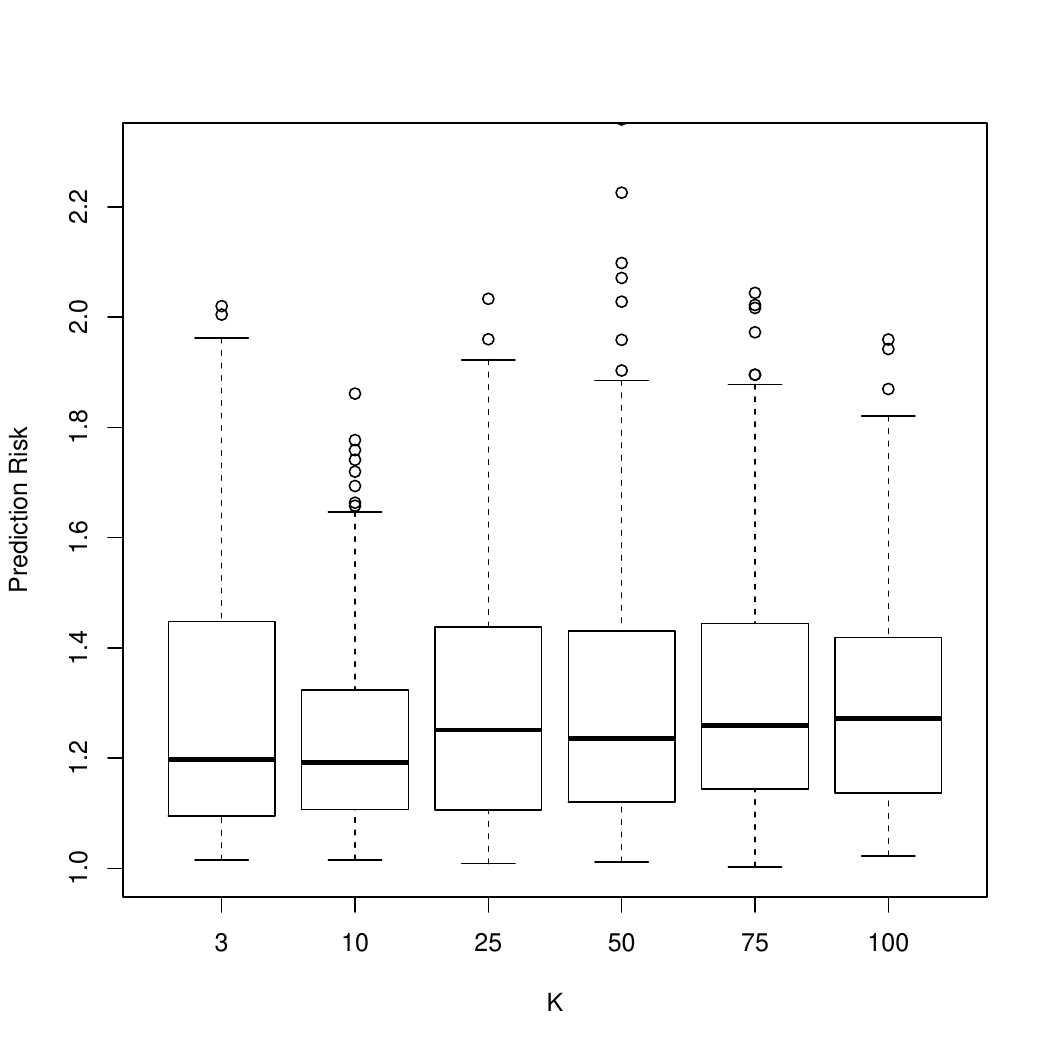}
	\caption*{$p = 1000$, $\alpha = 0.6$}
\end{subfigure}
\caption{\textbf{Prediction risk:}  The other parameters are set at: $\textrm{SNR} = 1$, $\rho= 0.9$.}
\label{fig:predRisk}
\end{figure}

For model selection, there is a more nuanced story.  For sensitivity, which describes how often
we would correctly identify a coefficient as nonzero, larger values of $K$ tend to work better.  For instance, 
in \autoref{fig:sensitivity}, we see that $K = 3$ is often decidedly worse than larger $K$, followed by $K = 10$.
As is often the case, this conclusion presents a trade-off with the
results for specificity (\autoref{fig:specificity}): smaller values of
$K$ tend to work better.  In general,
$\hat\beta(\hat\lambda)$ tends to have more nonzero entries as $K$
increases holding all else constant.  As the correlation parameter
$(\rho)$ or the signal to noise (SNR) increases, all values of $K$ have approximately the same performance.
%\textbf{Notes:} Keep 
%\begin{itemize}
%
%\item Pred risk $p=1000$, SNR = 1, $\rho = 0.9$ (all $\alpha$)
%\item Sensitivity $p=350$, SNR = 1, $\rho = 0.2$, $\alpha = .6$
%\item Sensitivity $p=1000$, SNR = 1, $\rho = 0.2$, $\alpha = .4,.6$
%\item Do same for specificity 
%\end{itemize}

\begin{figure} %This is SNR by Rho, with rho omitting 0.8
\centering
\begin{subfigure}[h]{.32\textwidth}  
	\includegraphics[width=1.5in,trim = 0 15 28 55,clip]
	{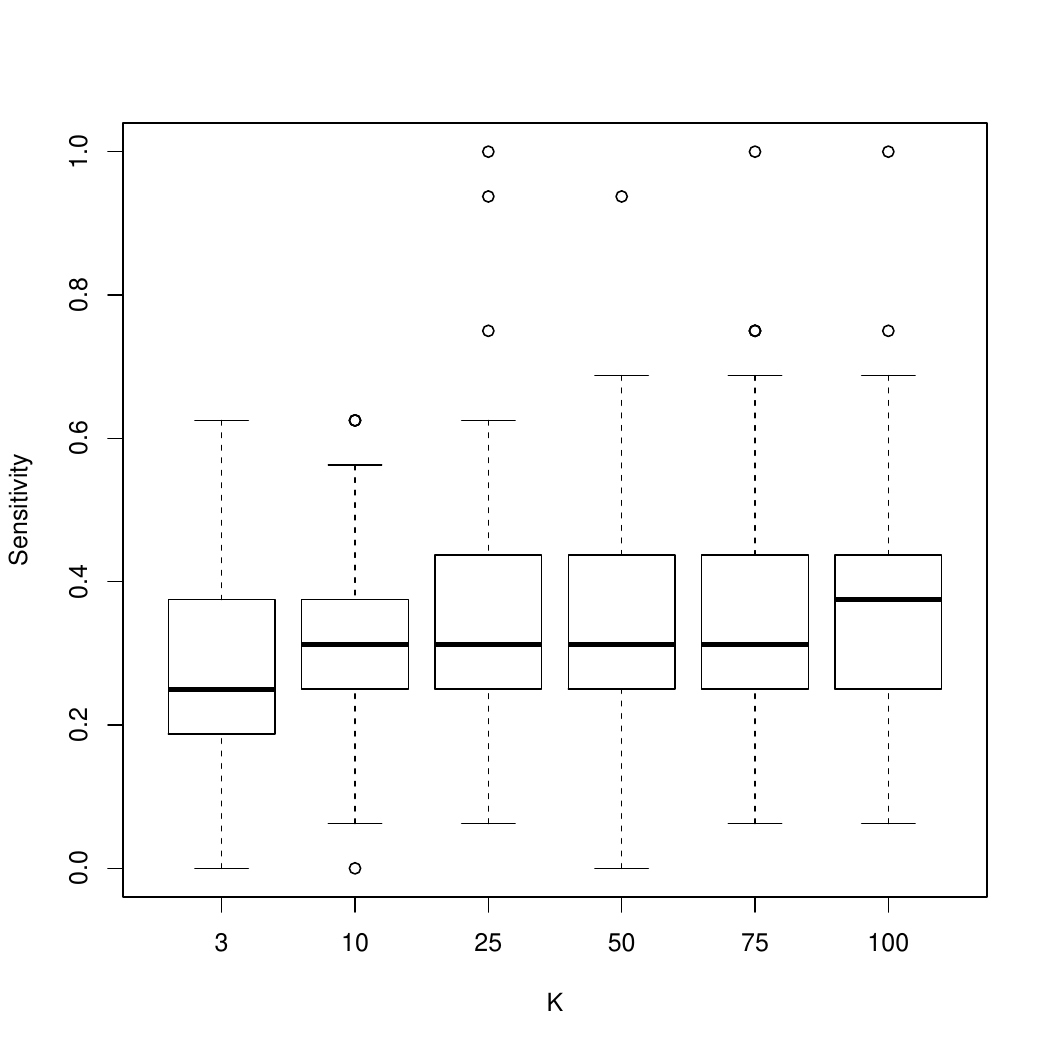}
	\caption*{$p = 350$, $\alpha = 0.6$}
\end{subfigure}
\begin{subfigure}[h]{.32\textwidth}  
	\includegraphics[width=1.5in,trim = 0 15 28 55,clip]
	{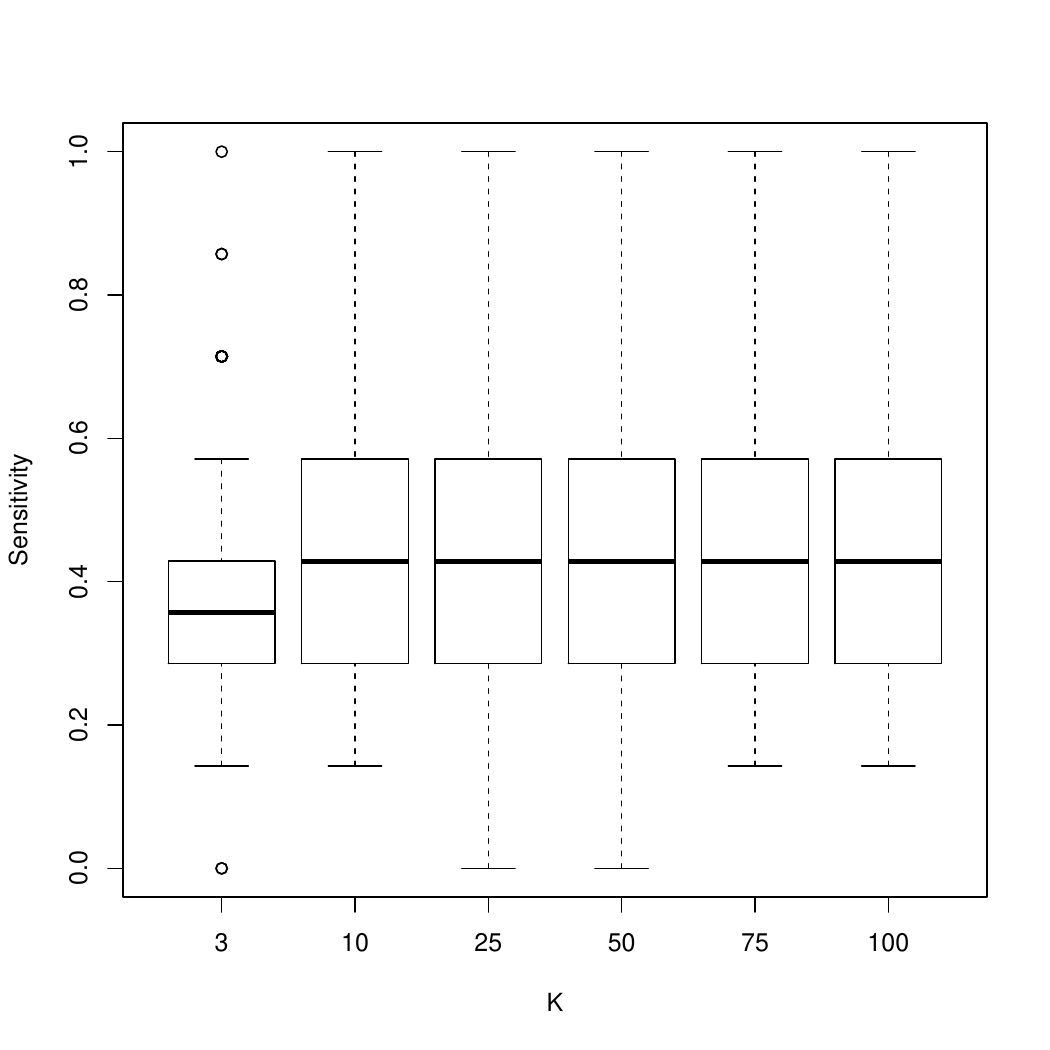}
	\caption*{$p = 1000$, $\alpha = 0.4$}
\end{subfigure}
\begin{subfigure}[h]{.32\textwidth}  
	\includegraphics[width=1.5in,trim = 0 15 28 55,clip]
	{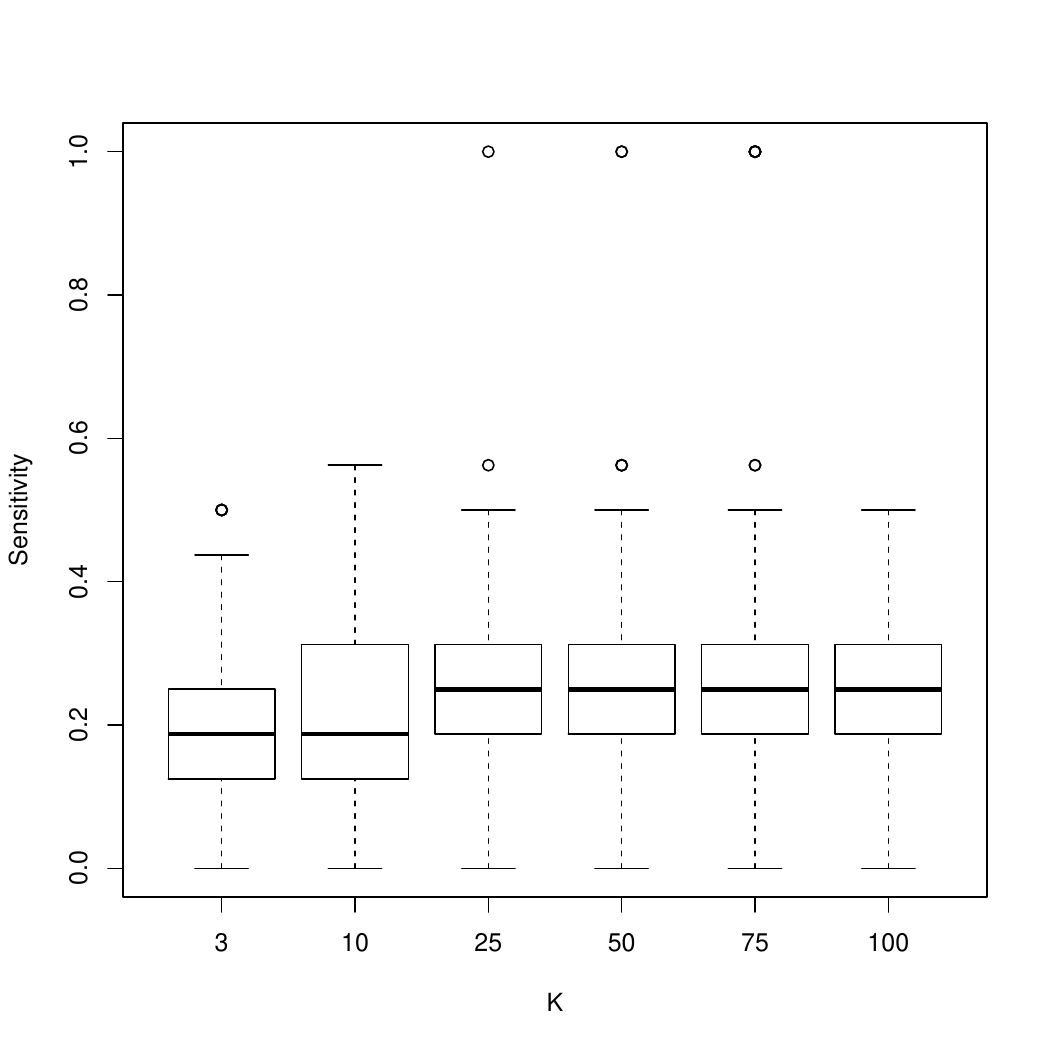}
	\caption*{$p = 1000$, $\alpha = 0.6$}
\end{subfigure}
\caption{\textbf{Sensitivity:} 
 The other parameters are set at: $\textrm{SNR} = 1$, 
$\rho= 0.2$.}
\label{fig:sensitivity}
\end{figure}

\begin{figure} %This is SNR by Rho, with rho omitting 0.8
\centering
\begin{subfigure}[h]{.32\textwidth}  
	\includegraphics[width=1.5in,trim = 0 15 28 55,clip]
	{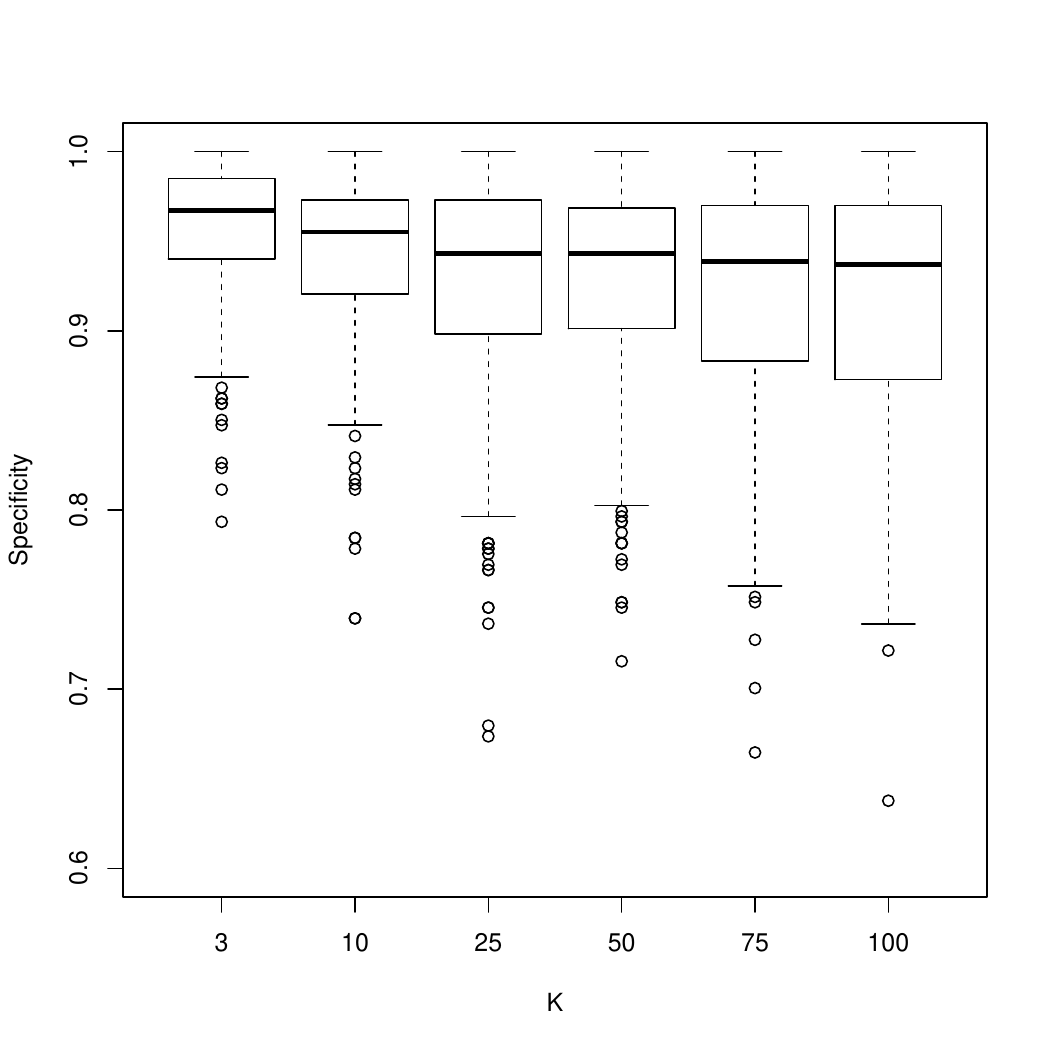}
	\caption*{$p = 350$, $\alpha = 0.6$}
\end{subfigure}
\begin{subfigure}[h]{.32\textwidth}  
	\includegraphics[width=1.5in,trim = 0 15 28 55,clip]
	{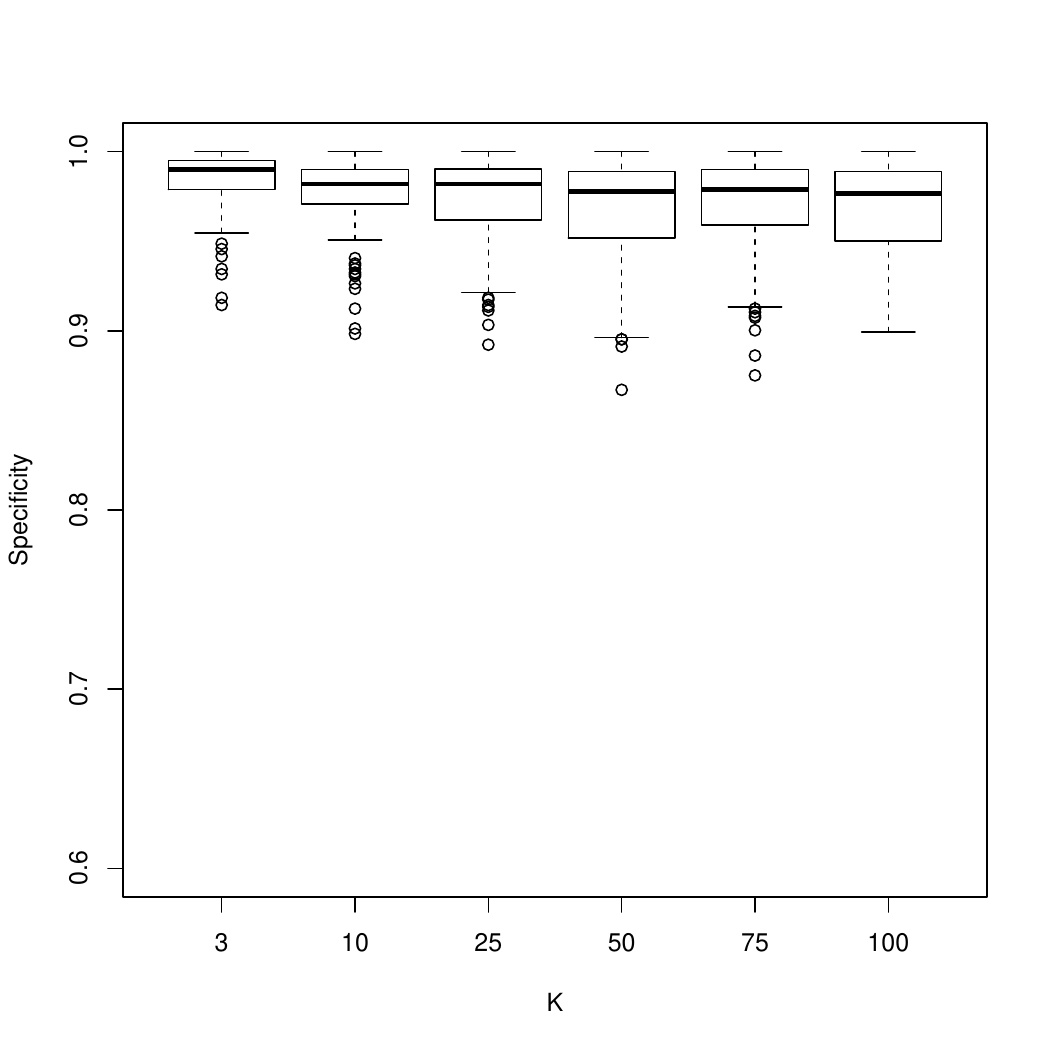}
	\caption*{$p = 1000$, $\alpha = 0.4$}
\end{subfigure}
\begin{subfigure}[h]{.32\textwidth}  
	\includegraphics[width=1.5in,trim = 0 15 28 55,clip]
	{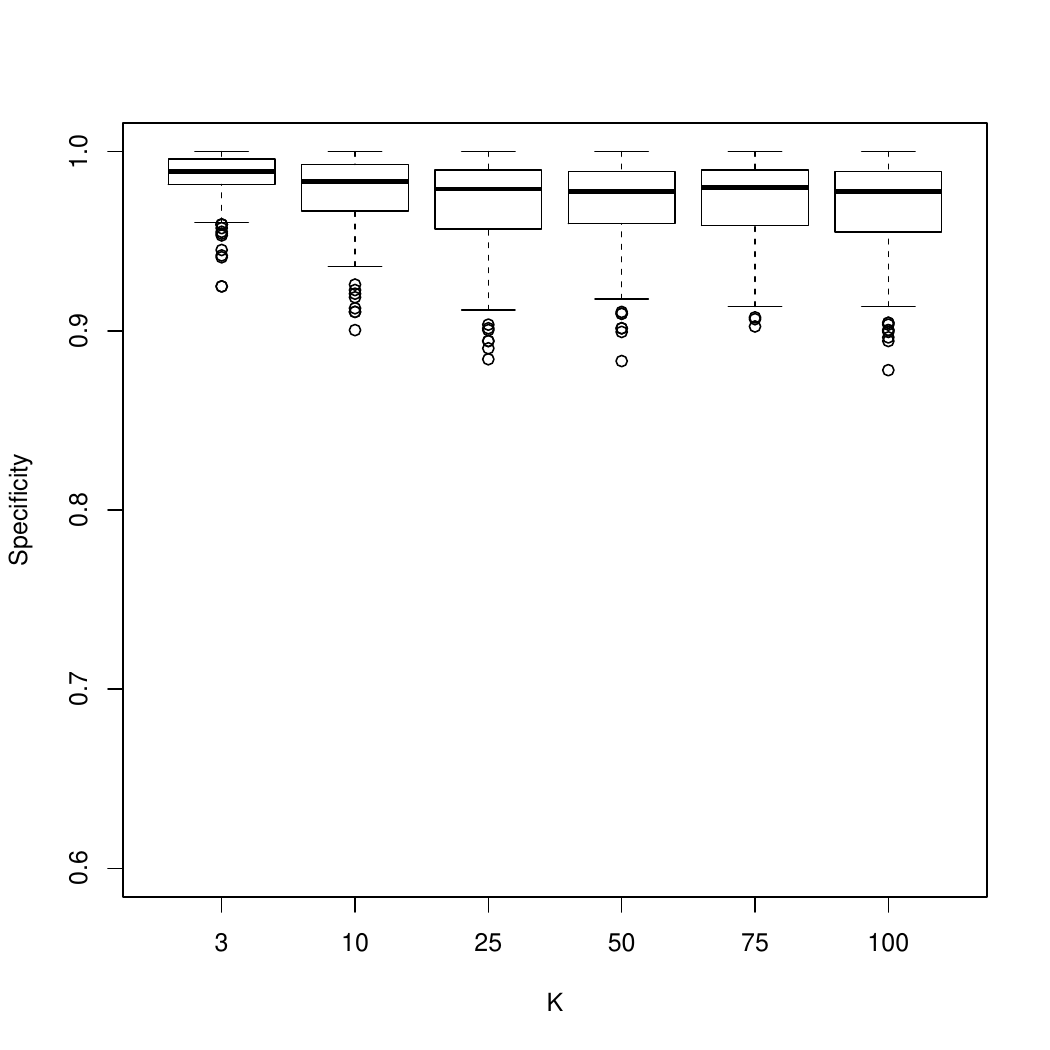}
	\caption*{$p = 1000$, $\alpha = 0.6$}
\end{subfigure}
\caption{\textbf{Specificity:}
 The other parameters are set at: $\textrm{SNR} = 1$, 
$\rho= 0.2$.}
\label{fig:specificity}
\end{figure}

\section{Discussion}
\label{sec:conclusion}

A common practice in data analysis is to attempt to estimate
a coefficient vector $\beta$
via lasso with its regularization parameter chosen via cross-validation.
Unfortunately, no definitive theoretical results 
existed as to the effect of choosing the tuning parameter in this data-dependent
way in the high-dimensional random design setting.

%\textcolor{red}{Remove or keep the next paragraph?}\\
%The main difficulties with our proof technique are handling
%$\cvrisk{V_n}{\S} $ and the choice of
%$T_n$.  Because of the
%complicated dependence between the held-out data and the estimators,
%our results depend on the {\em worst} predictor in some allowable
%set $\Sest$. Much faster rates should be possible if we can deal
%instead with the risk of best estimator directly. At the same time,
%as discussed in \autoref{sec:Tn}, in order to choose the allowable set
%$\Sest$, we must perform a grid search over possible values of
%$t$. Using $\tparmM$ as the upper bound on $T_n$ corresponds to the
%saturated model, but the whole 
%point of using any regularized model, is that the saturated model has
%high variance and low bias. It is our {\em a priori} belief that large
%values of $\tparm$ are unlikely to have good risk properties. Thus,
%imposing conditions on $\tparmM$ is again much stronger than
%desired. With a better understanding of these two issues, one would
%expect that we could achieve the same result in even more general
%settings. 

We have demonstrated that the lasso, group lasso,
and square-root lasso with tuning parameters chosen
by cross-validation are risk consistent relative to the linear oracle
predictor.
% , provided the size of the validation sets, 
% $c_n$, grows with the sample size $n$. 
Under strong conditions on the joint distribution (as in related
literature), our results differ from the standard convergence rates
by a factor of $\log n$. 
Imposing more mild
conditions on the joint distribution of the data, we can achieve the
same rate of convergence to the risk of the linear oracle predictor
with a data-dependent tuning parameter as 
with the optimal, yet inaccessible, oracle tuning parameter. Together,
these results show that if the 
data analyst is interested in using lasso for prediction, choosing the tuning parameter
via cross-validation is still a good procedure. In practice, our results
justify data-driven sets $\Lambda$ and $T$ over which we can safely
select tuning parameters. For $K$-fold CV, when $f^*$ is linear, our
theory suggests taking
$\Lambda=\left[(\log{p}\log{n}/n)^{1/2},\ \infty\right)$ while for $f^*$
arbitrary, a reasonable choice assuming $q=2$ is $T=\left[0,
\norm{Y}_2/(K n^{3/4}(\log{n}\log{p})^{1/2})\right)$.

%
%In our simulation study, CV is convincingly better than the other
%methods at minimizing prediction error. SSR performs marginally 
%better than CV when the following \emph{all} occur: the design
%correlation is low ($\rho = 0.2$), sparsity is high ($\alpha = 0.1$),
%and $p$ is not large---in other words, when the estimation problem is easier. 
%When SNR is very low, CV and AIC have essentially the same
%performance.  In all other cases, CV performs better  
%than all competitors.  As SNR, $\rho$, and $\alpha$ are all unknown in
%practice, using CV is a good choice due to its robustness 
%across these experimental conditions.
%
This work reveals some interesting open questions.  
First, our most general results do not apply for leave-one-out cross-validation as
$c_n =1$ for all $n$ and 
hence the upper-bounds become trivial. Leave-one-out cross-validation
is more computationally 
demanding than $K$-fold cross-validation, but is still used in
practice. These results also do not give any prescription for choosing
$K$ other than that it should be $o(n)$.  Our simulation study indicates that 
all $K$ ranging from 3 to $n$ have approximately the same prediction ability.
For model selection, larger $K$ tends to produce more nonzero coefficients and hence
has better sensitivity but poorer specificity.

As there are many other methods for choosing the tuning parameter in the lasso problem, 
a direct comparison of the behavior of the lasso estimator with
tuning parameter chosen via cross-validation versus a
degrees-of-freedom-based method is of substantial  
interest and should be investigated. Also, our results depend
strongly on the upper bound for $\Gen$ or the lower bound for
$\Lambda$. However, in most cases, we 
will never need to use tuning parameters this extreme. So it makes
sense to attempt to find more subtle theory to provide greater
intuition for the behavior of lasso under cross-validation.

% In this paper, we show that the lasso is risk consistent
% even when the tuning parameter  is
% selected via cross-validation. We feel that this paper provides
% some theoretical justification for the received wisdom that
% cross-validation is a useful tool for the applied researcher in the
% context of the lasso.

\section*{Acknowledgements} 
The authors gratefully acknowledge support from the National Science
Foundation (DH, DMS--1407543; DJM, DMS--1407439) and the Institute of
New Economic Thinking (DH and DJM, INO14-00020).
Additionally, the authors would like the thank Cosma Shalizi and Ryan Tibshirani for
reading preliminary  
versions of this manuscript and also Larry Wasserman for
the inspiration and encouragement. 

\bibliography{lassorefs}

\clearpage

\appendix

\section{Appendix}
\label{sec:proofSection}

To show~\autoref{thm:lasso}, we need a few preliminary
results. We will show how to rewrite the risk as a
quadratic form in \autoref{sec:squared-error-loss}.  In \autoref{sec:background}, we 
state a concentration of measure result and some standard properties
of Orlicz norms. In \autoref{sec:lemmas} 
we prove some useful preliminary lemmas.  Lastly, in \autoref{sec:mainProof} we prove
the main theorem and corollaries.

\subsection{Squared-error loss and quadratic forms}
\label{sec:squared-error-loss}

We can rewrite the various formulas for the risk from
\autoref{sec:notation-assumptions} as quadratic forms.   
Define the parameter to be $\gamma^{\top} := (-1,\beta^{\top})$,
with associated estimator $\hat\gamma_\gen^{\top} := (-1,\hat\beta_\gen^{\top})$.
We can rewrite \refeq{eq:risk-main} as
\begin{equation}
  \label{eq:risk-quad}
  \risk{\beta} = \E_{\mu_n} \left[ (\Yrv -\beta^\top\Xrv)^2\right]=
  \gamma^{\top} \Sn \gamma 
\end{equation} 
where $\Sigma_n := \E_{\mu_n}[ \Zrv \Zrv^{\top}]$.
Analogously, \refeq{eq:empirical-risk-main} has the following form
\begin{equation*}
  \label{eq:empirical-risk-quad} 
  \emprisk{\beta} = \frac{1}{n} \norm{Y- \mathbb{X}\beta}_2^2 =
  \gamma^{\top} \Sign \gamma, 
\end{equation*} 
where $\Sign = n^{-1}\sum_{i=1}^n Z_{i}Z_{i}^{\top}$.
Lastly, we rewrite \refeq{eq:cv-risk-main} as
\begin{equation}
  \label{eq:cv-risk-quad} 
  \cvrisk{V_n}{\gen} =  \frac{1}{K} \sum_{v \in V_n} (\cvgam_\gen)^T
  \cvsig{v} \cvgam_\gen, 
\end{equation} 
where $\cvsig{v}= |v|^{-1} \sum_{r \in v} Z_rZ_r^{\top}$, $\cvgam_\gen
:= (-1,\hat\beta^{(v)}_\gen)^{\top}$, and 
\begin{equation*}
  \label{eq:cv-estimator-quad}
  \hat\beta_\gen^{(v)} := \argmin_{\beta \in \B_\gen} \gamma^{\top} \Signor
  \gamma, 
\end{equation*}
with $\Signor := (n-|v|)^{-1}\sum_{r\notin v} Z_r Z_r^{\top}$.
%$\Signor := (n-1)^{-1}\sum_{i \neq r} Z_{i}Z_{i}^{\top}$.  

% \frac{1}{K} \sum_{v \in V_n}

\subsection{Background Results}
\label{sec:background} 
We use the following results in our proofs.
First is a special case of Nemirovski's inequality.
See 
%\citet*{Nemirovski2000} or 
\citet*{DumbgenVan-De-Geer2010} for more general formulations.
%\begin{lemma}[Nemirovski's inequality]
%  \label{lem:nemirovski}
%  Let $\xi_1,\ldots,\xi_n$ be independent random vectors in $\R^d$,
%  for $d\geq 3$ with $\E[\xi_i]=0$ and
%  $\E\norm{\xi_i}_2^2<\infty$. Then, for every $s \in [2,\infty]$ and index set $v$,
%  there exists an absolute constant $\widetilde{C}$ (independent of
%  $s$, $n$, $d$, validation sets $v$, and the distribution of the $\xi_i$'s) such that
%  \begin{equation*}
%    \E \norm{\sum_{i \in v} \xi_i}_s^2 \leq \widetilde{C}
%    \min\{s,\log d\} \sum_{i\in v} \E \norm{\xi_i}_s^2.
%  \end{equation*}
%\end{lemma}
\begin{lemma}[Nemirovski's inequality]
  \label{lem:nemirovski}
  Let $(\xi_i)_{i \in v}$ be independent random vectors in $\R^d$,
  for $d\geq 3$ with $\E[\xi_i]=0$ and
  $\E\norm{\xi_i}_2^2<\infty$. Then, for any validation set $v$ and distribution for the $\xi_i$'s,
  \begin{align*}
    \E \norm{\sum_{i \in v} \xi_i}_\infty^2 
     & \leq 
    (2e\log d - e) \sum_{i\in v} \E \norm{\xi_i}_\infty^2 
     \leq 
    2e\log d \sum_{i\in v} \E \norm{\xi_i}_\infty^2.
  \end{align*}
\end{lemma}

Also, we need the following results about the Orlicz norms.
\begin{lemma}[\citealt{vanweak}] 
\label{lem:orlicz}
%Main crux of proof uses this result (these are all from Weak Convergence and Empirical Processes, Chapter 2.2):
For any $\psi_r$-Orlicz norm with $1<r \in \mathbb{N}$ and sequence of $\R$-valued random variables 
$(\zeta_j)_{j=1,\ldots,m}$
\[
\norm{\max_{1\leq j \leq m}  \zeta_j}_{\psi_r} \leq \Psi \log^{1/r}(m + 1) \max_{1\leq j \leq m} \norm{ \zeta_j}_{\psi_r},
\]
where $\Psi$ is a constant that depends only on $\psi_r$.  
%This implies that 
%\[
%||X||_2^2 \leq 4(\log 2)^{-1 } ||X||_{\psi_2}^2.
%\]
\end{lemma}

\begin{lemma}[Corollary 5.17 in \citealt{Vershynin2012}]
  \label{lem:bernstein}
  Let $\xi_1,\ldots,\xi_n$ be iid centered random
  variables and let $\norm{\xi_i}_{\psi_1}\leq \kappa$. Then for every
  $\delta>0$,
  \[
  \P\left(\left|\sum_{i=1}^n \xi_i\right| \geq n\delta\right) \leq
  2\exp\left(-\cConst n
    \min\left\{\frac{\delta^2}{\kappa^2},\frac{\delta}{\kappa}\right\}\right), 
  \]
  where $\cConst=\cConstVal$.
\end{lemma}

\subsection{Supporting Lemmas}
\label{sec:lemmas}
Several times in our proof of the main results we
need to bound a quadratic form given by a symmetric matrix and an
estimator indexed by a tuning parameter.  To this end, we state the following lemma.
\begin{lemma}
  \label{lem:quad-form} 
  Suppose $a \in \mathbb{R}^p$ and $A \in
  \mathbb{R}^{p \times p}$.  Then
  \begin{equation*} 
    a^\top A a 
    \leq \norm{a}^2_{1}
    \norm{A}_{\infty},
  \end{equation*}
  where $\norm{A}_\infty:=\max_{i,j} |A_{ij}|$ is the entry-wise max norm.
\end{lemma}
% \begin{proof} (\autoref{lem:quad-form})
%   \begin{align*}
%     a^\top A a & \leq \norm{a}_{1} \norm {A a}_{\infty} 
%      \leq \norm{a}_{1} \max_{ij} |A_{ij}| \norm{a}_1  
%      = \norm{a}^2_{1} \norm{A}_{\infty},
%   \end{align*} where the first inequality follows by H{\"o}lder's
%   inequality.
% \end{proof}

We use \autoref{lem:nemirovski} to find the rate of
convergence for the sample covariance matrix to the population
covariance. 

\begin{lemma}
  \label{lem:leaveKout}
Let $v \subseteq \{1,2,\ldots,n\}$ be an index set and let $|v|$ be its number of elements. If $\mu \in \F_q$,
then there exists a constant 
$C$, depending only on $q$, such that
  \begin{equation*}
\E_\mu\norm{\cvsig{v} - \Sigma_n}_\infty 
\leq 
C\sqrt{\frac{(\log p)^{1 + 2/q} }{|v|}},
  \end{equation*}
  where it is understood that $2/\infty = 0$.
\end{lemma}
\begin{proof}(\autoref{lem:leaveKout})
  Let $\xi_r \in \mathbb{R}^{(p+1)^2}$ be the vectorized version of the zero-mean matrix 
  $\frac{1}{|v|}( Z_{r} Z_{r}^{\top} - \E \Zrv \Zrv^{\top} )$.  
  Then, by Jensen's inequality
  \begin{equation*}
    \left(\E\norm{\hat\Sigma_v - \Sigma_n}_\infty \right)^2
    \leq
    \E\norm{\hat\Sigma_v - \Sigma_n}_\infty^2
    = 
    \E\norm{\sum_{r\in v}  \xi_r}_{\infty}^2.
  \end{equation*}
  Using \autoref{lem:nemirovski} with $d = (p+1)^2$ and writing $\norm{X}_{L_2(\mu)}^2 := \E_\mu X^2$ we find
  \begin{align*}
\E_\mu \norm{\sum_{r\in v} \xi_r}_{\infty}^2 
    & \leq  2e\log\left( (p+1)^2\right) 
    \sum_{r \in v}  \bigg|\bigg|\norm{ \xi_r}_{\infty} \bigg|\bigg|_{L_2(\mu)}^2 \\
    & \leq  4e\log (p+1)
    \sum_{r \in v}  \left( \left(2(\log
          2)^{1/q-1}\right)\bigg|\bigg|\norm{ \xi_r}_{\infty}
      \bigg|\bigg|_{\psi_q} \right)^2 \\     
    & \lesssim \log (p+1)
    \sum_{r \in v}  \left( \log^{1/q}\left((p+1)^2 + 1\right)
        \bigg|\bigg|\norm{ \xi_r}_{\psi_q}\bigg|\bigg|_{\infty}
      \right)^2 \\     
    & \lesssim \log(p+1) \frac{1}{|v|^2}
    \sum_{r \in v}  \left( \log^{1/q}\left((p+1)^2 + 1\right) \Cf \right)^2 \\        
    & \leq  C'\log (p+1) \frac{1}{|v|^2}
    \sum_{r \in v}   \log^{2/q}\left((p+1)^2 + 1\right) \\        
    & \leq C\frac{(\log p)^{1 + 2/q} }{|v|}. \\        
  \end{align*}
Note that $\psi_q$ is the Orlicz norm induced by the measure $\mu_n$ and
the third inequality follows by 
\autoref{lem:orlicz}. 
\end{proof}

% \begin{lemma}
%   For any $\psi_r$ norm ($r\geq1$) , if $\norm{\xi}_{\psi_r}\leq \kappa<\infty$, then
%   $\norm{\xi-\E[\xi]}_{\psi_r} \leq 2\kappa$.
% \end{lemma}
% \begin{proof}
%   \begin{align}
%     \norm{\xi-\E[\xi]}_{\psi_r} &\leq
%                                   \norm{\xi}_{\psi_r}+\norm{\E[\xi]}_{\psi_r}\\
%     \norm{\E[\xi]}_{\psi_r} &= |\E[\xi]| \leq \E[|\xi|] \leq \norm{\xi}_{\psi_r}.
%   \end{align}
% \end{proof}

% \begin{lemma}
% If $\norm{f^*(X_i)}_{\psi_2}\leq F$ and
% $\norm{\epsilon_i}_{\psi_2}\leq \kappa$, then $\norm{Y_i^2}_{\psi_1}
% \leq 2(F + \kappa)^2$.
% \end{lemma}
% \begin{proof}
%   As $\psi_2$ is a norm, by the triangle inequality we have that 
%   \[
%   \norm{Y}_{\psi_2} = \norm{f^* + \epsilon}_{\psi_2} \leq
%   \norm{f^*}_{\psi_2} + \norm{\epsilon}_{\psi_2} \leq F + \kappa.
%   \]
%   Then, by the definition of $\psi_r$, we have $\norm{Y_i^2}_{\psi_1}
%   \leq 2(F + \kappa)^2$.
% \end{proof}

\begin{corollary}
  \label{cor:subExpY}
  By the definition of $\mu_n$,
  \[
  \P\left(\left|\norm{Y}_2^2-n\E [Y_1^2]\right| \geq n\delta\right)
  \leq 2\exp\left(-\cConst n
    \min\left\{\frac{\delta^2}{\Cf^{'2}},
      \frac{\delta}{\Cf'}\right\}\right),
  \]
  where $c=\cConstVal$ is an absolute constant and $\Cf'=(\log 2)^{1/q-1} \Cf$.
\end{corollary}
\begin{proof}
  This result follows immediately from \autoref{lem:bernstein} and the
  result \[|| \xi||_{\psi_1} \leq (\log 2)^{1/q-1 } ||
  \xi||_{\psi_q}\leq (\log 2)^{1/q-1 }\Cf = \Cf'.\]\end{proof}

\subsection{Proof of Main Results}
\label{sec:mainProof}
\begin{proof}[\autoref{thm:lasso}]

% Using the sets $D_n$ and $E_n$ defined in Lemmas
% \ref{lem:tmaxUpperBound} and  \ref{lem:tmaxLowerBound}, respectively,
% we can write the excess risk as 
Let $D_n,E_n$ be any two sets. 
Then we can make the following decomposition:
\begin{align}
\P\left( \mathcal{E}( \hgen,\ngen )  \geq \delta \right) 
& = 
\P\left( \mathcal{E}( \hgen,\ngen )  \geq \delta \cap D_n \cap E_n\right) + 
\P\left( \mathcal{E}( \hgen,\ngen )  \geq \delta \cap D_n^c \cap E_n\right) + \notag \\
& \quad + 
\P\left( \mathcal{E}( \hgen,\ngen )  \geq \delta \cap D_n \cap E_n^c\right) + 
\P\left( \mathcal{E}( \hgen,\ngen )  \geq \delta \cap D_n^c \cap E_n^c\right) \notag \\
& \leq
\P\left( \mathcal{E}( \hgen,\ngen )  \geq \delta \cap D_n \cap E_n\right) + 
2\P\left(D_n^c \right) + \P\left(E_n^c \right).
\label{eq:DnEn}
\end{align}
Also,
\begin{align} 
  \mathcal{E}( \hgen,\ngen ) 
  &= 
  \risk{\hat\beta_{\hgen}} - \risk{\beta_{\ngen}} \notag\\
  & = 
  \underbrace{\risk{\betat{\hgen}} -\cvrisk{V_n}{ \hgen}}_{(I)}+
  \underbrace{\cvrisk{V_n}{\hgen} -  \cvrisk{V_n}{\mgen}}_{(II)} + \notag\\
  & \qquad 
   + \underbrace{\cvrisk{V_n}{\mgen} - \emprisk{\betat{\ngen}}}_{(III)}
   + \underbrace{\emprisk{\betat{\ngen}} -
     \risk{\beta_{\ngen}}}_{(IV)}, \label{eq:decomposition} 
\end{align} 
where we use the notation $\betat{\gen} = \hat\beta(\B_{\gen})$.
Now, for any $\gen \in T_n$,
$
  \cvrisk{V_n}{\hgen} -  \cvrisk{V_n}{\gen}\leq 0,
$
by the definition of $\hgen$, and thus $(II) < 0$.

Let
$
D_n := \left\{ \tparm_{\max}  \leq 2n\Cf'/a_n \right\} $
and $  
E_n := \{ \tparm_{\max} \geq t_n  \}$.
On the set $E_n$, $(III) \leq \cvrisk{V_n}{\mgen} - \emprisk{\betat{\mgen}} =: \widetilde{(III)}$.
Taking the first term from \refeq{eq:DnEn} and combining it with the
decomposition in \refeq{eq:decomposition}, we 
see that 
\begin{align}
\P\left( \mathcal{E}( \hgen,\ngen )  \geq \delta \cap D_n \cap E_n\right)  
& \leq 
\P\bigg((I)  \geq \delta/3 \cap D_n\bigg)  
+ \P\bigg(\widetilde{(III)}  \geq \delta/3 \cap D_n \cap E_n\bigg)
  \notag\\&\quad +\P\bigg((IV) \geq \delta/3 \bigg) 
\label{eq:firstTerm}
\end{align}
We break the remainder of
this section into parts based on these terms.
% We will prove the result for the
% lasso first. The proof of \autoref{thm:lassoType} follows.

\paragraph{Final predictor and cross-validation risk $(I)$:}
\label{sec:final-pred-cross}

% The first piece is the difference between the loss of the final predictor $\risk{\betatcv}$
% and the cross-validation risk $\cvrisk{\, \tCV \;}$.  

Using the notation introduced in \autoref{sec:squared-error-loss}, 
note that by equations \eqref{eq:risk-quad} and \eqref{eq:cv-risk-quad}
\begin{align*}
  \risk{\betat{\hgen}} -\cvrisk{V_n}{\, \hgen}
  & = 
  \hatgamtT\, \Sigma_n\, \hatgamt 
  -\frac{1}{K}\sum_{v\in V_n} \loogamtT \cvsig{v} \loogamt \notag\\
  % \label{eq:star}
  & =
  \left[\hatgamtT\, \Sigma_n \, \hatgamt -
    \hatgamtT\left( \Sign \right) \hatgamt
  \right]  +%+ \\
  % \label{eq:starstar} 
  %& \qquad 
  \left[\hatgamtT\left( \Sign \right) \hatgamt-  
    \frac{1}{K}\sum_{v\in V_n} \loogamtT \cvsig{v} \loogamt \right]
\end{align*}

Addressing each of the terms in order,
\begin{align} 
  % \eqref{eq:star} cvsig{v}
  \left[\hatgamtT\, \Sigma_n \, \hatgamt -
    \hatgamtT\left(  \Sign \right) \hatgamt
  \right]
  &= 
  \hatgamtT\left( \Sigma_n - \Sign \right) \hatgamt %\notag \\
  \leq
  \norm{\hatgamt}_1^2 \norm{ \Sigma_n - \Sign}_\infty, \notag
  % & \leq 
  % \mnorm{\Sigma - \Sign}_{\infty}
  % \sup_{\tparm \in \Tparm} (1 + t)^2  \\ 
  % & \leq
  % \mnorm{\Sigma - \Sign}_{\infty} (1+ \tparm_{\max})^2. 
   \label{eq:IaBound}
\end{align}
where the inequality follows by \autoref{lem:quad-form}.
%  and the last two inequalities follow by the 
% definition of $\B_t$ and $T_n$, respectively.

Likewise,
\begin{align} 
  \lefteqn{\left[\hatgamtT\left( \Sign \right) \hatgamt- 
      \frac{1}{K}\sum_{v\in V_n} \loogamtT \cvsig{v} \loogamt \right]}\notag\\
  &= \left( \hatgamtT \Sign \hatgamt -  
    \frac{1}{K}\sum_{v\in V_n}\loogamtT \Sign \loogamt\right)%+\notag\\
  %&\quad
  + \left(  \frac{1}{K}\sum_{v\in V_n} \loogamtT \Sign\loogamt 
    -\frac{1}{K}\sum_{v\in V_n} \loogamtT \cvsig{v} \loogamt  \right) \notag\\
  &=  \frac{1}{K}\sum_{v\in V_n} \left( \hatgamtT
    \Sign \hatgamt - \loogamtT \Sign \loogamt\right)
  + \left(  \frac{1}{K}\sum_{v\in V_n}
    \loogamtT \left( \Sign-\cvsig{v}\right) \loogamt\right) \notag\\ 
  &\leq
  \frac{1}{K}\sum_{v\in V_n}\loogamtT \left( \Sign-\cvsig{v}\right) \loogamt.\notag
\end{align} 
The last inequality follows as $\hatgamt$ is chosen to
minimize $\hatgamtT \Sign \hatgamt$, and so for any $v \in V_n$,
\begin{equation} 
  \label{eq:empmin1}
  \hatgamtT \Sign \hatgamt \leq \loogamtT \Sign \loogamt.
\end{equation}

Continuing and using  \autoref{lem:quad-form},
\begin{align} 
  %\lefteqn{ 
    \frac{1}{K}\sum_{v\in V_n} \loogamtT \left(
      \Sign-\cvsig{v}\right) \loogamt%} \notag \\ 
  &  \leq  \frac{1}{K}\sum_{v\in V_n}   \norm{\loogamt}_1^2 \norm{ \Sign - \cvsig{v} }_\infty \notag \\ 
  &   \leq \frac{1}{K}\sum_{v\in V_n}  \norm{\loogamt}_1^2 \left(
    \norm{\hat{\Sigma}_n - \Sigma_n}_\infty + 
    \norm{\Sigma_n - \cvsig{v}}_\infty \right) \notag
       \label{eq:IbBound}
%  & = \sup_{\gen \in \Genn} \sup_{\beta \in \S_\gen} \norm{\gamma}_1^2
%  \left(\norm{\Sigma_n - \Sign}_{\infty} + 
%    \frac{1}{K_n}\sum_{v\in V_n} \norm{\Sigma_n - \cvsig{v}}_\infty \right)
\end{align}
Therefore,
\begin{align*} 
 (I)
  &\leq 
  \norm{\hat{\gamma}_{\tCV}}_1^2
 \norm{ \Sigma_n - \Sign}_\infty +   \frac{1}{K}\sum_{v\in V_n}\left(
   \hat \gamma_{\tCV}^{(v)}\right)^{\top} \left(
   \Sign-\cvsig{v}\right) \hat \gamma_{\tCV}^{(v)} \\
  &\leq
   (1+\tmax)^2  \bigg( 2 \norm{\Sigma_n -
      \Sign}_{\infty}  +  \frac{1}{K}\sum_{v\in V_n} \norm{\Sigma_n -
      \cvsig{v}}_\infty \bigg).
\end{align*}

By \autoref{lem:leaveKout} with $V_n = \{\{1,\ldots,n\}\}$ and $c_n = n$,
\[
  \E\norm{\Sigma_n-\Sign}_\infty \leq C_1\sqrt{\frac{(\log p)^{1 + 2/q} }{n}},
\]
while taking $V_n = \{v_1,\ldots, v_{K}\}$ shows,
\begin{equation*}
  \frac{1}{K}\sum_{v\in V_n} \E \norm{\Sigma_n - \cvsig{v}}_\infty 
  \leq 
  C_2\sqrt{\frac{(\log p)^{1 + 2/q} }{c_n}}.
\end{equation*}
% Furthermore, $\E\left[(1+t_{\max})^4\right] \asymp \E\left[(t_{\max})^4\right]$
% implies $\E\left[(1+t_{\max})^4\right]  \asymp (t_n^4)$.
Combining these two bounds together gives
%along Slutsky's theorem gives
%\begin{align} 
%(I) &= o_\P\left(\tmaxrate\right)\left( O_\P \left(\sqrt{\frac{\log
%      p}{n}}\right) + O_\P \left(\sqrt{\frac{\log
%      p}{c_n}}\right)\right) = o_\P\left( \tmaxrate \sqrt{\frac{\log
%      p}{c_n}}\right).
%\label{eq:IconclusionLasso}
%\end{align}
$$
\P\bigg((I)  \geq \frac{\delta}{3} \cap D_n\bigg)  
\leq 
\frac{3}{\delta}\E[(I) \, \mathbf{1}_{D_n} ] \notag\\
$$
\begin{equation}
\leq
\frac{3}{\delta}
\left(1 + \frac{2n(\log 2)^{1/q-1} C_q)}{a_n}\right)^2 
\bigg( 2C_1 \sqrt{\frac{(\log p)^{1 + 2/q}}{n}}  +  C_2  \sqrt{\frac{(\log p)^{1 + 2/q}}{c_n}}\bigg).
\label{eq:I}
\end{equation}

%%%%%%%%%%%%%%%%%%%%%%%%%%%%% 
\paragraph{Cross-validation risk and empirical risk $(III)$:}
%%%%%%%%%%%%%%%%%%%%%%%%%%%%% 
\label{sec:cross-valid-risk}
Due to the discussion following \refeq{eq:decomposition}, it is
sufficient to bound $\widetilde{(III)}$ instead. 
Recall that  
$\Signor = (n-c_n)^{-1}\sum_{r\notin v} Z_r Z_r^{\top}.$

Then,
\begin{align}
    \cvrisk{V_n}{\mgen} - \emprisk{\betat{\mgen}}
  &= \frac{1}{K}\sum_{v\in V_n}  \loogamT{\mgen} \cvsig{v} \loogam{\mgen} -
  \hatgamT{\mgen} \Sign \hatgam{\mgen} \notag \\
  &= \frac{1}{K}\sum_{v\in V_n}  \left( \loogamT{\mgen} \cvsig{v} \loogam{\mgen} - \loogamT{\mgen} \Signor
    \loogam{\mgen} \right)+ \notag \\
  &\qquad +
\frac{1}{K}\sum_{v\in V_n} \left( \loogamT{\mgen} \Signor \loogam{\mgen} -
    \hatgamT{\mgen} \Sign \hatgam{\mgen} \right) \notag \\
  & \leq \frac{1}{K}\sum_{v \in V_n} 
  \norm{\loogam{\mgen}}_1^2 \norm{\cvsig{v}-\Signor}_\infty, \notag
   \label{eq:IIIaBound}  
\end{align}
where the inequality follows by \autoref{lem:quad-form} and the fact that
$\loogam{\sgen}$ is chosen
to minimize $\loogamT{\sgen} \Signor \loogam{\sgen}$, which implies
\begin{equation}
  \label{eq:empmin2}
  \loogamT{\mgen} \Signor \loogam{\mgen}
  \leq \hatgamT{\mgen} \Signor \hatgam{\mgen}.
\end{equation}
As before, 
\[
\frac{1}{K}\sum_{v \in V_n} 
  \norm{\loogam{\mgen}}_1^2 \norm{\cvsig{v}-\Signor}_\infty  \leq
  (1+\tmax)^2  \frac{1}{K}\sum_{v \in V_n} \left( \norm{\Sigma_n - 
  \cvsig{v}}_{\infty}  +  \norm{\Sigma_n -
  \cvsig{(v)}}_\infty \right).
\]
We can use a
straight-forward adaptation of \autoref{lem:leaveKout} 
to show that
\[
\frac{1}{K}\sum_{v\in V_n} \E\norm{\Sigma_n-\Signor}_\infty 
\leq
C_3\sqrt{\frac{(\log p)^{1 + 2/q}}{n - c_n}}.
\]
Therefore, 
$$
\P\bigg(\widetilde{(III)}  \geq \delta/3 \cap D_n \cap E_n\bigg) 
 \leq 
\frac{3}{\delta}\E[\widetilde{(III)} \, \mathbf{1}_{D_n} ]
$$
\begin{equation}
\leq 
\frac{3}{\delta}
\left(1 + \frac{2n\Cf'}{ a_n}\right)^2 \left(C_2
\sqrt{\frac{(\log p)^{1 + 2/q}}{c_n}} + C_3 \sqrt{\frac{(\log p)^{1 +
        2/q}}{n - c_n}}\right). 
\label{eq:III}
\end{equation}
\paragraph{Empirical risk and expected risk $(IV)$} 
% \label{sec:empir-risk-expect}
The proof of these results is given by \citet*{GreenshteinRitov2004}. We
include a somewhat different proof for completeness.
Observe 
%the following bounds
%\[
%\emprisk{\betat{\ngen}} - \risk{\betat{\ngen}} \leq \sup_{\beta \in
%  \B_{\ngen}} \left| \risk{\beta} - \emprisk{\beta} \right| 
%\]
%and
\begin{align*}
  \risk{\betat{\ngen}} - \risk{\beta_{\ngen}} 
  & =
  \risk{\betat{\ngen}} - \emprisk{\betat{\ngen}} + 
  \emprisk{\betat{\ngen}} -  \risk{\beta_{\ngen}}  \\
& \leq
  \risk{\betat{\ngen}} - \emprisk{\betat{\ngen}}  + 
\emprisk{\beta_{\ngen}} -  \risk{\beta_{\ngen}} \\
& \leq
2\sup_{\beta \in \B_{\ngen}} \left| \risk{\beta} - \emprisk{\beta} \right|. 
\end{align*}
Using \autoref{lem:quad-form}
(See \autoref{sec:squared-error-loss} for notation)
\begin{equation*}
  \sup_{\beta \in \B_{\ngen}} \left| \risk{\beta} - \emprisk{\beta} \right|  
\leq 
\sup_{\beta\in \B_{\ngen}} \norm{\gamma}^2_1
  \norm{\Sign-\Sn}_\infty 
  \leq
(1+t_n)^2   \norm{\Sign-\Sn}_\infty .
%  & \leq (1+\ngen)^2  \norm{\Sign-\Sn}_\infty.
\end{equation*}
Therefore, 
\begin{equation}
\P\bigg((IV) \geq \delta/3 \bigg) 
 \leq 
\frac{3}{\delta}
\left(1 + t_n\right)^2 \bigg( C_4 \sqrt{\frac{(\log p)^{1 + 2/q}}{n}}\bigg).
\label{eq:IV}
\end{equation}

The proof follows by combining \refeq{eq:DnEn} with \refeq{eq:firstTerm} and using
the bounds from Equations \eqref{eq:I}, \eqref{eq:III}, and \eqref{eq:IV}.

Lastly, there are the various constants incurred in the course of this proof.  In \autoref{lem:orlicz},
the constant $\Psi$ can be chosen arbitrarily small based on inspecting the proof of Lemma 2.2.2 in
\citet{vanweak}.  As this constant premultiplies every term in $\Omega_{n,1}$ and $\Omega_{n,2}$,
we can without loss of generality set the constant equal to
one. For instance, in \autoref{lem:leaveKout}, the constant $C$ is upper bounded by
$8e^{1/2}(\log 2)^{(2-q)/(2q)}8^{1/q}C_q\Psi$.  This constant can be taken arbitrarily small
by choosing $\Psi$ small enough.
\end{proof}

\begin{lemma}
  \label{lem:tmaxUpperBound}
  Define the set 
  $
  D_n := \left\{ \tparm_{\max}  \leq 2n\Cf'/a_n \right\},
  $
  where $a_n$ is the normalizing rate defined in \autoref{cor:lasso}
  and $\Cf'=\Cf(\log 2)^{1/q-1}$. Then, 
  $\P(D^c_n) \leq e^{-\cConst n}$.
  \end{lemma}
\begin{proof} 
  By \autoref{cor:subExpY},
  \begin{align*}
    \P\left(\frac{\norm{Y}_2^2}{a_n}
    \geq \frac{n(\E [Y_1^2]+\delta)}{a_n} \right)
    &\leq \exp\left(-\cConst n 
   \min\left\{\frac{\delta^2}{\Cf^{'2}}, 
   \frac{\delta}{\Cf'}\right\}\right).
  \end{align*}
  Furthermore, $\E[Y_1^2] \leq \Cf'$, so 
  \begin{align*}
        \P\left(\frac{\norm{Y}_2^2}{a_n}
    \geq \frac{n(\Cf'+\delta)}{a_n} \right) 
    &\leq \P\left(\frac{\norm{Y}_2^2}{a_n} 
    \geq \frac{n(\E [Y_1^2]+\delta)}{a_n} \right).
  \end{align*}
  Therefore, setting $\delta=\Cf'$ yields
    $\P(\norm{Y}_2^2/a_n
    \geq 2nh/a_n ) 
    \leq e^{-\cConst n}.
  $
\end{proof}

\begin{lemma}
  \label{lem:tmaxLowerBound}
  Define the set 
  $
  E_n := \left\{ \tparm_{\max}  \geq t_n \right\}.
  $
  If $a_nt_n=o(n)$, then, for all $n>2a_nt_n/\E[Y_1^2]$,
  $
  \P(E^c_n) \leq \exp\left\{-\cConst n\right\}.
  $
\end{lemma}
\begin{proof}
  By \autoref{cor:subExpY},
  \begin{align*}
    \P\left(\frac{\norm{Y}_2^2}{a_n}
    \leq \frac{n(\E [Y_1^2]-\delta)}{a_n} \right)
    &\leq \exp\left(-\cConst n 
   \min\left\{\frac{\delta^2}{\Cf^{'2}}, 
   \frac{\delta}{\Cf'}\right\}\right),
  \end{align*}
  for all $\delta>0$. Setting $\delta=\E[Y_1^2]-a_nt_n/n$ therefore
  implies
  \begin{align*}
    \P\left(\tmax
    \leq t_n \right)
    &\leq \exp\left(-\cConst n 
   \min\left\{\frac{(\E[Y_1^2]-a_nt_n/n)^2}{\Cf^{'2}}, 
   \frac{\E[Y_1^2]-a_nt_n/n}{\Cf'}\right\}\right) \\
    &\leq \exp\left(-\cConst n 
   \min\left\{\frac{\E[Y_1^2]^2}{4\Cf^{'2}}, 
   \frac{\E[Y_1^2]}{2\Cf'}\right\}\right).
  \end{align*}
  Since $0<\E[Y_1^2]/ \Cf'\leq 1$, the result follows.
\end{proof}

\begin{proof}[\autoref{cor:lassolike}]
  For the $\sqrt{\mbox{lasso}}$, the result
  is nearly immediate as we are considering the same constraint set
  $\norm{\beta}_1\leq t$ and the same search space for the
  tuning parameter $T=[0,\ \norm{Y}_2^2/a_n]$. However, in
  Equations \eqref{eq:empmin1} and \eqref{eq:empmin2}, we rely on the empirical
  minimizer. The analogous results here are $\hatgamtT \Sign
    \hatgamt \leq \loogamtT \Sign \loogamt$ and $
  \loogamT{\mgen} \Signor \loogam{\mgen}
  \leq \hatgamT{\mgen} \Signor \hatgam{\mgen}$ respectively,
  but this implies that \eqref{eq:empmin1} and \eqref{eq:empmin2}
  hold. 
  
  For the group lasso with
  $\max_g\sqrt{|g|}=O(1)$, we have
  $
  t\geq \sum_{g \in G} \sqrt{|g|} \norm{\beta_g}_2 \geq  \norm{\beta}_1
  $
  so that \autoref{lem:tmaxLowerBound} still applies with $\tmax$ as
  before. We note that in this case, the oracle group linear model is
  restricted to the ball $\B_{t_n}=\{\beta : \sum_{g \in G} \sqrt{|g|}
  \norm{\beta_g}_2 \leq t_n\}$ rather than the larger set $\{\beta:
  \norm{\beta}_1 \leq t_n\}$.
\end{proof}

%\begin{proof}(\autoref{thm:lassoType})
%Returning now to equations  \eqref{eq:IaBound} through \eqref{eq:VandVIaBound} and substituting
%$u$ for $\gen$, we can form analogous bounds to the ones in the proof of \autoref{thm:lasso}.
%To this end, note that $\norm{\gamma}_1^2 \leq (1 + \norm{\beta}_1)^2$,  $(1 + \norm{\beta}_1)^2 \asymp \norm{\beta}_1^2$, $\norm{\beta}_1 \leq  \sum_{g\in G} \sqrt{|g|}\norm{(\beta_j)_{j \in g}}_2$, and $\max_{g \in G} |g| = a_n$.  Therefore,
%\begin{align*}
%\E \norm{\hat \gamma_{\uCV}}_1^4 
%& \leq \E\left(1 + \norm{\hat \beta_{\uCV}}_1\right)^4  \\
%& \asymp \E\norm{\hat \beta_{\uCV}}_1^4  \\
%& \leq  \E \left( \sum_{g\in G} \sqrt{|g|}\norm{(\hat \beta_{\uCV,j})_{j \in g}}_2 \right)^4\\
%& \leq  a_n^2\E\left(\sum_{g\in G}\norm{(\hat \beta_{\uCV,j})_{j \in g}}_2\right)^4 \\
%& \leq  a_n^2\E u_{\max}^4 \leq a_n^2u_n^4.
%\end{align*}
%The rest of the proof follows from the proof of \autoref{thm:lasso}.
%\end{proof}

\end{document}